\documentclass[12pt]{article}

\newtheorem{lemma}{Lemma}
\newtheorem{theorem}{Theorem}

\newtheorem{remark}{Remark}

\usepackage{amsmath,amssymb,ascmac,color,ulem}

\begin{document}
\title{ 
 Flat coordinates of Frobenius prepotentials  related with 
the reflection groups of types $H_3$ and $H_4$}

\author{Rei Aradachi\thanks{Institute of Science and Engineering, Kanazawa University, Japan.\\
\hspace{0.5cm}
Email address: reiaradachi@gmail.com} \and 
Hiromasa Nakayama\thanks{College of Bioresource Sciences, Nihon University, Japan.\\
\hspace{0.5cm}
Email address:nakayama.hiromasa@nihon-u.ac.jp} \and
Jiro Sekiguchi\thanks{Tokyo University of Agriculture and Technology, Japan.\\
\hspace{0.5cm}
Email address: sekiguti@cc.tuat.ac.jp
}}

\maketitle

\begin{abstract}
In this article, we first explain a group  theoretic  interpretation of the
derivation of the relation
between 
the flat coordinates of the polynomial prepotential $(H_3)$
and those of the algebraic prepotential $(H_3)'$ given in \cite{KMS2}
constructed by M. Feigin, D. Valeri and J. Wright \cite{FVW}.
By the same idea explained in the case of $(H_3)$,
we will show a relation between the flat coordinates
of the polynomial prepotential $(H_4)$ and
those of the algebraic prepotential $H_4(9)$ given in \cite{Se}.
\end{abstract}

\begin{flushright}



\end{flushright}

\section{Introduction}
We begin this introduction with recalling two prepotentials
related with the Coxeter group $W(H_3)$ of type $H_3$.
One is the polynomial  prepotential 
$F_{H_3}(x_1,x_2,x_3)$  
defined by B. Dubrovin (cf. \cite{Du}):
\begin{equation}
\label{equation:H3-poly-potential}
F_{H_3}(x_1,x_2,x_3)=\frac{{x_1}^{11}}{3960}+\frac{{x_1}^5 {x_2}^2}{20}+
\frac{{x_1}^2 {x_2}^3}{6}+\frac{1}{2} \left({x_1} {x_3}^2+{x_2}^2
   {x_3}\right).
\end{equation}
The other is the algebraic prepotential
$F_{(H_3)'}(t_1,t_2,t_3)$ (cf. \cite{KMS2}).
Its concrete form is given by
\begin{equation}
\label{equation:(H3)'-potential}
F_{(H_3)'}(t_1,t_2,t_3)=\displaystyle{\frac{t_1t_3^2+t_2^2t_3}{2}-
\frac{t_1^4z}{18}-\frac{7t_1^3z^4}{72}-\frac{17t_1^2z^7}{105}-\frac{2t_1z^{10}}{9}
-\frac{64z^{13}}{585}},
\end{equation}
where $z$ is an algebraic function of $t_1,t_2$ defined by the equation
$$
t_2+t_1z+z^4=0.
$$
Let $\Delta_{H_3}(x_1,x_2,x_3)$ (resp. $\Delta_{(H_3)'}(t_1,t_2,t_3))$
be the discriminant of $F_{H_3}$ (resp. $F_{(H_3)'}$).
Then $\Delta_{H_3}$ is a weighted homogeneous polynomial
with weights $1/5,3/5,1$.
On the other hand,
we define ${\tilde \Delta}_{(H_3)'}(t_1,t_3,z)=
\Delta_{(H_3)'}(t_1,-t_1z-z^4,t_3)$.
Then
${\tilde \Delta}_{(H_3)'}$ is a weighted homogeneous polynomial
of $t_1,t_3,z$ with  weights $3/5,1,1/5$.
These two facts suggest the existence of an isomorphism
between
 $\Delta_{H_3}$ and ${\tilde \Delta}_{(H_3)'}$.
Actually, as was shown in \cite{Se}, 
${\tilde \Delta}_{(H_3)'}$ coincides with
 $\Delta_{H_3}$ up to a constant factor,
 by the transformation
\begin{equation}
\label{equation:H3-(H3)'-trans}
t_1=m^3(x_1^3+x_2),\>t_3=m^5\left(\frac{3}{2}x_3+x_1^2x_2\right),\>
z=-mx_1\> (t_2=m^4x_1x_2).
\end{equation}
Since $(x_1,x_2,x_3)$ is the flat coordinate of $F_{H_3}$ and
since $(t_1,t_2,t_3)$ is those of $F_{(H_3)'}$,
(\ref{equation:H3-(H3)'-trans}) is regarded as a relation
between these two flat coordinates.

Feigin, Valeri and Wright (\cite{FVW}) gave a relation between
these two flat coordinates by a different method.
The relation they obtained is
\begin{equation}
\label{equation:H3-(H3)'-trans-FVW}
\left\{\begin{array}{lll}
x_1&=& \frac{c_0}{m^3} t_1, \\
x_2&=& -\frac{c_0^3}{27m^9} (t_1 + z^3) (5 t_1^2 + 180 t_3 z + 10 t_1 z^3 + 32 z^6),\\
x_3 &=& -\frac{c_0^5}{810 m^{15} }(-119 t_1^5 - 12000 t_3^3 - 
         4200 t_1^3 t_3 z + 
                6050 t_1^4 z^3 - 3150 t_1^2 t_3 z^4 \\
&&+ 18745 t_1^3 z^6 + 
         3600 t_1 t_3 z^7 + 21040 t_1^2 z^9 + 
                3360 t_3 z^{10} + 10640 t_1 z^{12 }\\
&&+ 2176 z^{15})\\
\end{array}
\right.
\end{equation}
for a certain constant $c_0$.

The two formulas   (\ref{equation:H3-(H3)'-trans}) and
(\ref{equation:H3-(H3)'-trans-FVW}) appear different,
but the latter can be derived from the former through an idea
based on the existence of a three dimensional
non-reflection
representation $V$ of $W(H_3)$.
Such a representation is realized by three polynomials
$P_1,P_2,P_3$ 
of coordinates of $V$ of degree three.
The polynomials $P_1,P_2,P_3$ play an important role in our argument.
The concrete forms of  $P_1,P_2,P_3$ and the property of them
will be given in \S2.4.

Let $F_{H_4}$ be the polynomial prepotential of type $H_4$ (see \cite{Du}).
There are seven algebraic prepotentials denoted by
$H_4(j)\>(j=1,2,3,4,6,7,9)$ (cf. \cite{Se}) related with
the Coxeter group $W(H_4)$ of type $H_4$.
Feigin, Valeri and Wright \cite{FVW} determined 
a relation between the flat coordinates of the polynomial
prepotential $F_{H_4}$  and those
of each of algebraic prepotentials $F_{H_4(j)}\>(j=1,2,3,4,7)$.
It is noted here that there is an obstruction to
apply the method in \cite{FVW} to the remaining two cases 
 $(F_{H_4},F_{H_4(6)})$, $(F_{H_4},F_{H_4(9)})$.
The argument for the case $(F_{H_3},\>F_{(H_3)'})$ of \S2 will be treated below
is applicable to the case $(F_{H_4},F_{H_4(9)})$, in spite that
it isn't to the case $(F_{H_4},F_{H_4(9)})$.
The main purpose of this paper is to demonstrate that, in the case 
$(F_{H_4},F_{H_4(9)})$, it is indeed possible to obtain a relation
between the flat coordinates  applying the idea performed in \S2.
We shall treat this case in Section 3.
In Subsection 3.1 to Subsection 3.3, we prepare some notation and show elementary properties
of the Coxeter group $W(H_4)$.
Let $(u_1,u_2,u_3,u_4)$ be a linear coordinate of the reflection
representation space and let $Z_1,Z_2,Z_3,Z_4$ be basic $W(H_4)$-invariants
of degrees 2, 12, 20, 30, respectively.
Subsection 3.4 treats a four dimensional representation of
$W(H_4)$ on the space of polynomials of 7th degree.
The existence of such a representation plays an important role in the
subsequent consideration.
The four dimensional representation looks like a reflection representation of $W(H_4)$
but different.
By using this representation, it s possible to construct four $W(H_4)$-invariants
denoted by $Y_1,Y_2,Y_3,Y_4$ of degrees 14, 84, 140, 210, respectively.
It is a bit heavy job to express these invairants 
as polynomials of the basic $W(H_4)$-invariants
whose degrees are known to be 2,12,20,30.
This is shown in Theorem \ref{theorem:formula-Y-Z}
which is  an analogue of Lemma \ref{lemma:lemma-3} for the case $W(H_3)$.
After the preparation in Subsections 3.6 and 3.7,
we shall treat the main topic of this paper in Subsection 3.7.
We first introduce the polynomial prepotential $F_{H_4}(x_1,x_2,x_3,x_4)$
and the algebraic prepotential $F_{H_4(9)}(t_1,t_2,t_3,t_4)$ and their
discriminants $\Delta_{H_4}$ and $\Delta_{H_4(9)}$.
It is noted here that $F_{H_4(9)}$ contains an additional algebraic function $w_0$
defined by the algebraic equation
$$
\frac{3}{20}t_1^2-t_3w_0^2+\frac{1}{5}t_2w_0^4+\frac{8}{5}t_1w_0^7+w_0^{14}=0.
$$
By this relation, $t_3$ is expressed by a rational function of $t_1,t_2,t_4,w_0$
and as a consequence, the  $\Delta_{H_4(9)}$
is regarded as a function of $t_1,t_2,t_4,w_0$,
eliminating $t_3$ by the reason above.
At present, we observe that $w_0^{10}\Delta_{H_4(9)}$ is a polynomial of $t_1,t_2,t_4,w_0$.
Then it is possible to show that by the transformation
$$
\left\{\begin{array}{lll}
t_1&=&\frac{1}{6}x_1(4{x_1}^6+315 x_2),\\
t_2&=&\frac{3675}{8}({x_1}^4x_2+18x_3),\\
t_4&=&\frac{46305}{32}(5{x_1}^3{x_2}^2+8{x_1}^5x_3+90x_4),\\
w_0&=&-x_1,\\
\end{array}
\right.
$$
$w_0^{10}\Delta_{H_4(9)}$ coincides with $\Delta_{H_4}$
up to a constant factor.
This is also an analogue of the formula 
(\ref{equation:H3-(H3)'-trans}) for the case $(H_3,(H_3)')$.
As a consequence, we obtain the required formula for the case $(H_4,H_4(9))$
analogous to 
(\ref{equation:identification-yast-x}).
The result is extremely lengthy.

Closing this introduction, we give a remark.
Reading \cite{FVW} leads us the problems which we are going to explain.
Let $Y_1,Y_2,Y_3,Y_4$ be $W(H_4)$-invariants of degrees 14, 84, 140, 210,
respectively
introduced above.
It follows from  the main result of this paper  that
$w_0^{k_{j}}Y_j\>(j=1,2,3,4)$ are polynomials of $t_1,t_2,t_4,w_0$
for some non-negative integers $k_1,k_2,k_3,k_4$.

\vspace{5mm}
Problem 1: Determine the algebraic equation for $w_0$ whose
coefficients  are polynomials of $Y_1,Y_2,Y_3,Y_4$.

Problem 2: Express $t_1,t_2,t_4$ as functions of $Y_1,Y_2,Y_3,Y_4,w_0$.

\vspace{5mm}
It is too difficult for the authors to solve these two problems.

\section{The case  $W(H_3)$}

\subsection{Two realizations of $W(H_3)$ as matrix groups}

We first recall the definition of the Coxeter group $W(H_3)$  of type $H_3$ and its  elementary properties
(cf. \cite{YS2}).
It is known that $W(H_3)$ is the group generated by
three elements
$s_1,s_2,s_3$ with the relations:
$$
s_1^2=s_2^2=s_3^2=1,\quad
(s_1s_2)^5=(s_2s_3)^3=(s_1s_3)^2=1.
$$
The group $W(H_3)$ is realized as a real reflection group.
I explain this briefly.
We put $e_1=(1,0,0),e_2=(0,1,0),e_3=(0,0,1)$ and denote by  $V$ the vector space
over ${\bf R}$ spanned by  $e_1,e_2,e_3$.
We introduce three vectors
\begin{equation}
\label{equation:def-root-H3}
\alpha_1=e_1,\>\alpha_2=-\frac{1}{2}(ae_1+\bar{a}e_2+e_3),\>\alpha_3=e_3,
\end{equation}
where
\begin{equation}
\label{equation:def-a-bara}
a=\frac{1+\sqrt{5}}{2},\>\bar{a}=\frac{1-\sqrt{5}}{2}.
\end{equation}
We denote by ${\tt u}=(u_1,u_2,u_3)$ the linear coordinate of $V$.
Let ${\tilde s}_j$ be the reflection of $V$ corresponding to $\alpha_j\>(j=1,2,3)$.
Then it is easy to show that
$$
{\tilde s}_1:u_1\to -u_1,\>{\tilde s}_3:u_3\to -u_3.
$$
As to ${\tilde s}_2$, by direct computation, we find that
$$
\left\{
\begin{array}{lll}
{u_1}&\to&\frac{1}{2} ({\bar a} {u_1}+u_2-a {u_3}),\\
{u_2}&\to&\frac{1}{2} (u_1+{ a} {u_2}-{\bar a} {u_3}),\\
{u_3}&\to&\frac{1}{2} (-au_1-{\bar a} {u_2}+u_3).\\
   \end{array}
   \right.
$$
Then  $W(H_3)$ is realized as the reflection group
generated by $\rho(s_j)\>(j=1,2,3)$, where
$$
\rho(s_1)=\left(\begin{array}{ccc}
-1&0&0\\
0&1&0\\
0&0&1\\
\end{array}
\right),\>
\rho(s_2)=\frac{1}{2}\left(\begin{array}{ccc}
{\bar a}&1&-a\\
1&a&-{\bar a}\\
-a&-{\bar a}&1\\
\end{array}
\right),\>
\rho(s_3)=\left(\begin{array}{ccc}
1&0&0\\
0&1&0\\
0&0&-1\\
\end{array}
\right).
$$
In the sequel,
we put ${\cal W}(H_3)=\rho(W(H_3))$.

We define another representation of $W(H_3)$ which we are going to construct.
Let  $\rho^{\ast}(s_j)=\rho(s_j)|_{\sqrt{5}\to-\sqrt{5}}\>(j=1,2,3)$ be three matrices.
Namely, $\rho^{\ast}(s_j)=\rho(s_j)\>(j=1,3)$ and
$$
\rho^{\ast}(s_2)=\frac{1}{2}\left(\begin{array}{ccc}
a&1&-{\bar a}\\
1&{\bar a}&-a\\
-{\bar a}&-a&1\\
\end{array}\right).
$$
Then ${\cal W}^{\ast}(H_3)$ is the group generated by $\rho^{\ast}(s_j)\>(j=1,2,3)$.
It is underlined here that ${\cal W}(H_3)$ and ${\cal W}^{\ast}(H_3)$ are isomorphic but not conjugate,
that is, there is no $g\in GL(3,{\bf C})$ such that 
${\cal W}^{\ast}(H_3)=g({\cal W}(H_3))g^{-1}$.

\subsection{Basic invariants of $W(H_3)$}
We start with introducing the elementary symmetric polynomials and difference product of $u_1^2,u_2^2,u_3^2$:
$$
\left\{
\begin{array}{lll}
{\epsilon_1}&=&
 {u_1}^2+{u_2}^2+{u_3}^2,\\
   {\epsilon_2}&=& {u_1}^2 {u_2}^2+{u_1}^2 {u_3}^2+{u_2}^2 {u_3}^2,\\
  {\epsilon_3}&=& {u_1}^2 {u_2}^2 {u_3}^2,\\
\delta&=& \left({u_1}^2-{u_2}^2\right) \left({u_1}^2-{u_3}^2\right) \left({u_2}^2-{u_3}^2\right).\\
  \end{array}
  \right.
  $$
  Let $P({\tt u})$ be a polynomial of ${\tt u}=(u_1,u_2,u_3)$.
 Then $P({\tt u})$ is ${\cal W}(H_3)$-invariant  if
 $$
P((u_1,u_2,u_3)\rho(s_k))=P({\tt u})\>(k=1,2,3).
$$
It is straightforward to show that
\begin{equation}
\label{equation:def-I-ep}
\begin{array}{lll}
I_1 ({\tt u})&=&2\epsilon_1,\\
I_2 ({\tt u})&=&20(-11\epsilon_3+\epsilon_1\epsilon_2+\sqrt{5}\delta),\\
I_3 ({\tt u})&=&80(95\epsilon_2\epsilon_3-32\epsilon_1^2\epsilon_3-5\epsilon_1\epsilon_2^2+
2\epsilon_1^3\epsilon_2-\frac{1}{25}\epsilon_1^5+3\sqrt{5}\epsilon_2\delta)\\
\end{array}
\end{equation}
are basic invariants  of the polynomial ring ${\bf C}[u_1,u_2,u_3]$ under the action of $W(H_3)$  in the above sense.
Similarly it is possible to find basic ${\cal W}^{\ast}(H_3)$-invariant polynomials.
We can take
\begin{equation}
\label{equation:def-J-ep}
\begin{array}{lll}
J_1 ({\tt u})&=&2\epsilon_1,\\
J_2 ({\tt u})&=&20(-11\epsilon_3+\epsilon_1\epsilon_2-\sqrt{5}\delta),\\
J_3 ({\tt u})&=&80(95\epsilon_2\epsilon_3-32\epsilon_1^2\epsilon_3-5\epsilon_1\epsilon_2^2+
2\epsilon_1^3\epsilon_2-\frac{1}{25}\epsilon_1^5-3\sqrt{5}\epsilon_2\delta).\\
\end{array}
\end{equation}
as ${\cal W}^{\ast}(H_3)$-basic invariants.
It is noted here that
$J_k$ is obtained from $I_k$ by changing $\sqrt{5}$ to $-\sqrt{5}$.

\subsection{The polynomial potential $(H_3)$}
The  polynomial potential $F_{H_3}$ is defined by B. Dubrovin 
(cf. (\ref{equation:H3-poly-potential})).
We explain briefly  how to define the discriminant  
$\Delta_{H_3}(x_1,x_2,x_3)$ of $F_{H_3}$.
Let $C_{H_3}$ be the $3\times 3$ matrix such that
$(C_{H_3})_{ij}=\partial_{x_i}\partial_{x_{4-j}}F_{H_3}$.
Then
$\Delta_{H_3}=\det(C_{H_3})$.
By a direct computation, we find that
\begin{equation}
\label{equation:H3-disc}
\begin{array}{lll}
\Delta_{H_3}(x_1,x_2,x_3)&=&-\frac{{x_1}^{15}}{1000}-\frac{{x_1}^{12} {x_2}}{100}-\frac{{x_1}^{10} {x_3}}{100}+\frac{2 {x_1}^9 {x_2}^2}{25}+\frac{{x_1}^6
   {x_2}^3}{50}+\frac{{x_1}^5 {x_3}^2}{10}-\frac{6}{5} {x_1}^4 {x_2}^2 {x_3}\\
&&   +\frac{23 {x_1}^3 {x_2}^4}{25}
   +{x_1}^2 {x_2} {x_3}^2-\frac{9}{5}
   {x_1} {x_2}^3 {x_3}+\frac{27 {x_2}^5}{125}+{x_3}^3.\\
   \end{array}
\end{equation}
By the substitution of $x_j=I_j\>(j=1,2,3)$, the discriminant 
$\Delta_{H_3}$ coincides with
the square of the product of the 15  linear forms 
below up to a constant factor:
{\footnotesize
$$
\begin{array}{llll}
u_1,&u_2,&u_3,&\\
u_1+au_2+{\bar a}u_3,&
u_1+au_2-{\bar a}u_3,&
u_1-au_2+{\bar a}u_3,&
u_1-au_2-{\bar a}u_3,\\
{\bar a}u_1+u_2+au_3,&
-{\bar a}u_1+u_2+au_3,&
{\bar a}u_1+u_2-au_3,&
-{\bar a}u_1+u_2-au_3,\\
au+{\bar a}u_2+u_3,&
au-{\bar a}u_2+u_3,&
-au+{\bar a}u_2+u_3,&
-au-{\bar a}u_2+u_3,\\
  \end{array}
$$
}
Let $\ell_j({\tt u})$ be the $j$-th linear form of the above 15 forms.
These linear forms correspond to the 15 reflections of ${\cal W}(H_3)$.
We define
$$
D_{H_3}({\tt u})=\prod_{j=1}^{15}\ell_j({\tt u}).
$$
Then 
$$
\Delta_{(H_3)}(x_1,x_2,x_3)=-2^{15}\cdot 5^2D_{H_3}({\tt u})^2
$$
 under the identification $x_k=I_k\>(k=1,2,3)$.

Let $\ell_j^{\ast}({\tt u})$ be the linear form obtained from $\ell_j({\tt u})$
by changing $\sqrt{5}$ to $-\sqrt{5}$.
For example,
since 
$$\ell_{7}({\tt u})=  u_1-au_2-{\bar a}u_3,$$
we have
$$\ell_{7}^{\ast}({\tt u})=  u_1-{\bar a}u_2-au_3.$$

\subsection{Three dimensional representation of ${\cal W}(H_3)$ on the space of polynomials of 3rd degree}

We  introduce three polynomials of degree 3:
\begin{equation}
\label{equation:P1P2P3-H3}
\begin{array}{lll}
P_1({\tt u})&=&
{u_1} \left( {u_1}^2-3{\bar a}{u_2}^2-3a{u_3}^2\right),\\
P_2({\tt u})&=&{u_2} \left(-3a {u_1}^2+
   {u_2}^2-3 {\bar a}{u_3}^2\right),\\
P_3({\tt u})&=&   {u_3} \left(-3 {\bar a} {u_1}^2-3a {u_2}^2+{u_3}^2\right)\\
\end{array}
\end{equation}
and define 
$$
v_k=P_k({\tt u})\>(k=1,2,3)
$$
and 
$${\tt v}=(v_1,v_2,v_3).
$$

The  formula in the lemma below plays a basic role in the subsequent consideration:

\begin{lemma}
\label{lemma:lemma-1}
$$
(P_1({\tt u}\rho(s_j)),P_2({\tt u}\rho(s_j)),P_3({\tt u}\rho(s_j)))={\tt v}\rho^{\ast}(s_j)\quad(j=1,2,3).
$$
\end{lemma}

This lemma shows that the representation $\rho^{\ast}$ of ${\cal W}(H_3)$ 
is realized on the vector space 
${\tilde V}=\{{\tt v}=(v_1,v_2,v_3)|\>v_j\in {\bf R}\>(j=1,2,3)\}$ over ${\bf R}$.
As we noted that
$
J_k({\tt v})\>(k=1,2,3)
$ 
are  the basic invariants of ${\cal W}^{\ast}(H_3)$ on the polynomial ring 
${\bf C}[v_1,v_2,v_3]$,
Lemma \ref{lemma:lemma-1} shows that
$J_k(P_1({\tt u}),P_2({\tt u}),P_3({\tt u}))\>(k=1,2,3)$ are
also ${\cal W}(H_3)$-invariant.

\begin{lemma}
\label{lemma:lemma-3}
If $v_k=P_k({\tt u})\>(k=1,2,3)$, then
\begin{equation}
\label{equation:J{ast}I}
\left\{
\begin{array}{lll}
{J_1({\tt v})}&=& c_0  \left({I_1}^3+3 {I_2}\right),\\
{J_2({\tt v})}&=& -c_0^3 {I_2} \left(3 {I_1}^6-30 {I_1}^3{I_2}-60 {I_1} {I_3}+5 {I_2}^2\right),\\
   {J_3({\tt v})}&=&- \frac{1}{10} c_0^5 \left(15 {I_1}^{12} {I_2}+30 {I_1}^{10} {I_3}-105
   {I_1}^9 {I_2}^2+300 {I_1}^7 {I_2} {I_3}-1140 {I_1}^6 {I_2}^3\right.\\
   &&\left.+150 {I_1}^4 {I_2}^2 {I_3}-4545 {I_1}^3 {I_2}^4-6000 {I_1}^2
   {I_2} {I_3}^2+4200 {I_1} {I_2}^3 {I_3}\right.\\
   &&\left.-357 {I_2}^5-4000 {I_3}^3\right).\\
   \end{array}
   \right.
  \end{equation}
for a non-zero constant $c_0$.
\end{lemma}

This lemma is shown by direct computation.

We define
$$
D_{H_3}^{\ast}({\tt u})=\prod_{j=1}^{15}\ell_j^{\ast}({\tt u}).
$$
Then as in the case of $\Delta_{(H_3)}$,
we find that
$$
\Delta_{(H_3)}(x_1^{\ast},x_2^{\ast},x_3^{\ast})=
-\frac{25}{512}D_{H_3}^{\ast}({\tt v})^2
$$
under the identification $x_k^{\ast}=J_k({\tt v})\>(k=1,2,3)$.
Moreover,
we write
$x_j=I_j\>(j=1,2,3)$.
Then (\ref{equation:J{ast}I}) turns out to be
\begin{equation}
\label{equation:X->I}
\left\{
\begin{array}{lll}
{x_1^{\ast}}&=& c_0  \left({x_1}^3+3 {x_2}\right),\\
{x_2^{\ast}}&=& -c_0^3 {x_2} \left(3 {x_1}^6-30 {x_1}^3{x_2}-60 {x_1} {x_3}+5 {x_2}^2\right),\\
   {x_3^{\ast}}&=&- \frac{c_0^5}{10}  \left(15 {x_1}^{12} {x_2}+30 {x_1}^{10} {x_3}-105
   {x_1}^9 {x_2}^2+300 {x_1}^7 {x_2} {x_3}-1140 {x_1}^6 {x_2}^3\right.\\
   &&\left.+150 {x_1}^4 {x_2}^2 {x_3}-4545 {x_1}^3 {x_2}^4-6000 {x_1}^2
   {x_2} {x_3}^2+4200 {x_1} {x_2}^3 {x_3}\right.\\
   &&\left.-357 {x_2}^5-4000 {x_3}^3\right).\\
   \end{array}
   \right.
  \end{equation}

\subsection{Algebraic potential $(H_3)'$}

In this section, we treat
 the algebraic prepotential $(H_3)'$ defined in (\ref{equation:(H3)'-potential}).
We show its elementary property which will be 
used in the subsequent consideration.

Let $\Delta_{(H_3)'}(t_1,t_2,t_3)$ be the discriminant of $F_{(H_3)'}$ which is defined by a similar way to the case $\Delta_{H_3}$.
Its concrete form is given by
$$
\begin{array}{ll}
&\Delta_{(H_3)'}(t_1,t_2,t_3)\\
=&72 t_1^5 + 3000 t_3^3 + 1800 t_1^3 t_3 z - 1500 t_1 t_3^2 z^2 - 
     1460 t_1^4 z^3 + 1050 t_1^2 t_3 z^4 - 2400 t_3^2 z^5 \\
&- 
  4325 t_1^3 z^6 - 
     2400 t_1 t_3 z^7 - 4440 t_1^2 z^9 - 1920 t_3 z^{10} - 2240 t_1 z^{12} - 
  512 z^{15}.\\
\end{array}
$$
The discriminant $\Delta_{(H_3)'}(t_1,t_2,t_3)$ is not a polynomial of $t_1,t_2,t_3$
by the reason that it contains $z$ which is an algebraic function of $t_1,t_2$.
Noting that $t_2=-t_1z-z^4$,
we define
${\tilde \Delta}_{(H_3)'}(t_1,t_3,z)=\Delta_{(H_3)'}(t_1,-t_1z-z^4,t_3)$.
Then
${\tilde \Delta}_{(H_3)'}(t_1,t_3,z)$ is a weighted homogeneous polynomial
of $t_1,t_3,z$ with  weights $3/5,1,1/5$.
It is known that
the discriminant $\Delta_{(H_3)}(x_1,x_2,x_3)$ is a weighted homogeneous polynomial of 
$x_1,x_2,x_3$ with weights $1/5,3/5,1$.
These facts suggest the existence of  a relation between 
$\Delta_{(H_3)}(x_1,x_2,x_3)$  and 
${\tilde \Delta}_{(H_3)'}(t_1,t_3,z)$
which will be given in the lemma below.

\begin{lemma}
Suppose that by a non-trivial weighted homogeneous transformation
between $(t_1,t_3,z)$ and $(x_1,x_2,x_3)$ defined by
\begin{equation}
\label{equation:weight-map-H3}
t_1=a_1x_1^3+a_2x_2,\>t_3=b_1x_3+b_2x_1^2x_2+b_3x_1^5,\>z=c_1x_1,
\end{equation}
${\tilde \Delta}_{(H_3)'}(t_1,t_3,z)$ coincides with
$\Delta_{H_3}(x_1,x_2,x_3)$ up to a constant factor,
where $a_i,b_j,c_1$ are constants.
Then there is a non-zero constant $m$ such that
(\ref{equation:weight-map-H3}) turns out to be
\begin{equation}
\label{equation:weight-map-H3-re}
t_1=m^3(x_1^3+x_2),\>t_3=m^5(\frac{3}{2}x_3+x_1^2x_2),\>z=-mx_1.
\end{equation}
\end{lemma}

This lemma is proved by a direct computation.

The formula (\ref{equation:weight-map-H3-re}) is equivalent to
\begin{equation}
\label{equation:weight-map-H3-2}
x_1=-\frac{1}{m}z,\>x_2=\frac{1}{m^3}(t_1+z^3),\>x_3=\frac{1}{6m^5}(2t_3-t_1z^2-z^5).
\end{equation}

\begin{theorem}
\label{theorem:theorem-1}
If $x_j^{\ast}\>(j=1,2,3)$ are defined by (\ref{equation:X->I}), then
\begin{equation}
\label{equation:No-16}
\left\{\begin{array}{lll}
x_1^{\ast} &=& \frac{c_0}{m^3} t_1, \\
x_2^{\ast}&=& -\frac{c_0^3}{27m^9} (t_1 + z^3) (5 t_1^2 + 180 t_3 z + 10 t_1 z^3 + 32 z^6),\\
x_3^{\ast} &=& -\frac{c_0^5}{810 m^{15} }(-119 t_1^5 - 12000 t_3^3 - 
         4200 t_1^3 t_3 z + 
                6050 t_1^4 z^3 - 3150 t_1^2 t_3 z^4 \\
&&+ 18745 t_1^3 z^6 + 
         3600 t_1 t_3 z^7 + 21040 t_1^2 z^9 + 
                3360 t_3 z^{10} + 10640 t_1 z^{12 }\\
&&+ 2176 z^{15}).\\
\end{array}
\right.
\end{equation}
\end{theorem}

This theorem is  a consequence of
(\ref{equation:weight-map-H3-2}).

There is an interesting relation between $D_{H_3}({\tt u})$ and $
D_{H_3}^{\ast}(P_1,P_2,P_3)$ which is given in the lemma below.

\begin{lemma}
\label{lemma:reflection-reflection-ast}

(i)
$q_j({\tt u})=\ell_j^{\ast}(P_1({\tt u}),P_2({\tt u}),P_3({\tt u}))/\ell_j({\tt u})$
is a polynomial of ${\tt u}$ for $j=1,2,\ldots,15$.

(ii) We put
$$
\begin{array}{lll}
Q_0(I_1,I_2,I_3)&=&{I_1}^{15}-5 {I_1}^{12} {I_2}-20 {I_1}^{10} {I_3}-485 {I_1}^9 {I_2}^2-600 {I_1}^7 {I_2}
   {I_3}-1775 {I_1}^6 {I_2}^3\\
   &&-400 {I_1}^5 {I_3}^2+8700 {I_1}^4 {I_2}^2 {I_3}-12800 {I_1}^3
   {I_2}^4+14000 {I_1}^2 {I_2} {I_3}^2+8000 {I_3}^3.\\
   \end{array}
   $$
Then
\begin{equation}
\label{equation:jacobian-H3}
D_{H_3}^{\ast}(P_1,P_2,P_3)=c_1Q_0(I_1,I_2,I_3)D_{H_3}({\tt u}).
\end{equation}
for a non-zero constant $c_1$.
\end{lemma}

{\bf Proof}.
We fist recall an elementary identity:
$$
\frac{\partial(J_1,J_2,J_3)}{\partial(u_1,u_2,u_3)}=
\frac{\partial(I_1,I_2,I_3)}{\partial(u_1,u_2,u_3)}\frac{\partial(J_1,J_2,J_3)}{\partial(I_1,I_2,I_3)}
=
\frac{\partial(P_1,P_2,P_3)}{\partial(u_1,u_2,u_3)}
\frac{\partial(J_1,J_2,J_3)}{\partial(P_1,P_2,P_3)}.
$$
It is known (cf. \cite{YS2}) that
the determinant of
the Jacobian matrix
$
\displaystyle{\frac{\partial(I_1,I_2,I_3)}{\partial(u_1,u_2,u_3)}}
$
is the same as $D_{H_3}({\tt u})$ up to a constant factor.
By the same reason, that of
$
\displaystyle{\frac{\partial(J_1,J_2,J_3)}{\partial(P_1,P_2,P_3)}}
$
is the same as $D_{H_3}^{\ast}({\tt v})$ up to a constant factor.
As a consequence, 
$$
D_{H_3}(u_1,u_2,u_3)\cdot 
\det\left(\frac{\partial(J_1,J_2,J_3)}{\partial(I_1,I_2,I_3)}
\right)
=
D_{H_3}^{\ast}(P_1,P_2,P_3)\cdot 
\det\left(
\frac{\partial(P_1,P_2,P_3)}{\partial(u_1,u_2,u_3)}
\right).
$$
It follows from (\ref{equation:P1P2P3-H3}) that
$\displaystyle{ \det\left(\frac{\partial(P_1,P_2,P_3)}{\partial(u_1,u_2,u_3)}\right)}$
coincides with $(I_2-I_1^3)$ up to a constant factor.
On the other hand, 
it follows from Lemma
\ref{lemma:lemma-3} that
$\displaystyle{\det\left(\frac{\partial(J_1,J_2,J_3)}{\partial(I_1,I_2,I_3)}\right)}$
coincides with
 $(I_2-I_1^3)Q_0(I_1,I_2,I_3)$ up to a constant factor.

Then the lemma follows. []

\subsection{Relation with the paper of Feigin, Valeri and Wright \cite{FVW}}

This section is devoted to an interpretation of the result by
 Feigin, Valeri and Wright \cite{FVW} on the construction of a map from the flat coordinates of
 the prepotential $F_{(H_3)'}$ to those of the prepotential $F_{H_3}$.
 
We start with preparing their notation.
\begin{equation}
\label{equation:def-y-ep}
\left\{
\begin{array}{lll}
y_1&=&\epsilon_1.\\
y_2&=&-11\epsilon_3+\epsilon_1\epsilon_2+\sqrt{5}\delta,\\
y_3&=&95\epsilon_2\epsilon_3-32\epsilon_1^2\epsilon_3-5\epsilon_1\epsilon_2^2+
2\epsilon_1^3\epsilon_2+3\sqrt{5}\epsilon_2\delta,\\
\end{array}
\right.
\end{equation}
The relation between $I_k$ in
(\ref{equation:def-I-ep})
and $y_1,y_2,y_3$ are
\begin{equation}
\label{equation:identification-I-y}
I_1=2y_1,\>I_2=20y_2,\>I_3=\frac{16}{5}(25y_3-y_1^5).
\end{equation}

They constructed a map from $(t_1,t_2,t_3)$-space to $(y_1,y_2,y_3)$-space by
\begin{equation}
\label{equation:def-y-t-a}
\left\{
\begin{array}{lll}
     {y_1}&=& \frac{20}{3}t_1,\\
 {y_2}&=& \frac{3200}{729} \left(180
   {t_2} {t_3}-32 {t_2}^2 z-22 {t_1} {t_2} z^2-5 {t_1}^3-5 {t_1}^2
   z^3\right),\\
{y_3}&=& \frac{128000}{19683} \left(12000 {t_3}^3-3360 {t_2}^2 {t_3} z^2-3390 {t_1}^2
   {t_2} {t_3}-3120 {t_1} {t_2} {t_3} z^3\right.\\
&&\left.+810 {t_1}^3 {t_3} z
   +4112 {t_1} {t_2}^3+2176 {t_2}^3 z^3-2176 {t_1}^2 {t_2}^2 z-119 {t_1}^3 {t_2}
   z^2\right.\\
&&\left.+200 {t_1}^5+119 {t_1}^4 z^3\right).\\
   \end{array}
   \right.
\end{equation}

We are going to show that (\ref{equation:def-y-t-a})
is equivalent to (\ref{equation:No-16})
 in Theorem \ref{theorem:theorem-1} by changing $y_j$ with 
different functions $y_j^{\ast}\>(j=1,2,3)$ defined below 
(cf. (\ref{equation:identification-yast-x})).

At first, we note that
since the expression of (\ref{equation:def-y-t-a}) is not unique because of the relation 
   $
t_2+t_1z+z^4=0,
$
we eliminate $t_2$ in (\ref{equation:def-y-t-a})
 by $t_2=-t_1z-z^4$. Then
(\ref{equation:def-y-t-a}) turns out to be
\begin{equation}
\label{equation:def-y-t-b}
\left\{
\begin{array}{lll}
  {y_1}&=& \frac{20 }{3}t_1,\\
  {y_2}&=& -\frac{3200}{729} \left({t_1}+z^3\right) \left(180 {t_3}
   z+5 {t_1}^2+10 {t_1} z^3+32 z^6\right),\\
{y_3}&=& \frac{128000}{19683} \left(12000 {t_3}^3+4200 {t_1}^3 {t_3} z+3150 {t_1}^2
   {t_3} z^4-3600 {t_1} {t_3} z^7-3360 {t_3} z^{10}\right.\\
   &&\left.+200 {t_1}^5-6050 {t_1}^4
   z^3-18745 {t_1}^3 z^6-21040 {t_1}^2 z^9\right.\\
   &&\left.-10640 {t_1} z^{12}-2176
   z^{15}\right),\\
   \end{array}
   \right.
   \end{equation}

Since $x_j=I_j\>(j=1,2,3)$, the relation (\ref{equation:identification-I-y}) is equivalent to
\begin{equation}
\label{equation:identification-y-x}
y_1=\frac{1}{2}x_1,\>y_2=\frac{1}{20}x_2,\>y_3=\frac{1}{800}(10x_3+x_1^5).
\end{equation}
We introduce $y_1^{\ast},y_2^{\ast},y_3^{\ast}$ by
\begin{equation}
\label{equation:identification-y-x-ast}
y_1^{\ast}=\frac{1}{2}x_1^{\ast},\>y_2^{\ast}=\frac{1}{20}x_2^{\ast},
\>y_3^{\ast}=\frac{1}{800}(10x_3^{\ast}+{x_1^{\ast}}^5).
\end{equation}
By virtue of the equation of Theorem 1
 and (\ref{equation:identification-y-x-ast}),
we have
\begin{equation}
\label{equation:identification-yast-x}
\left\{
\begin{array}{lll}
y_1^{\ast}&=&\frac{1}{2m}c_0 t_1, \\
 y_2^{\ast}&=&-\frac{1}{540m^3} 
   c_0^3 (t_1 + z^3) (5 t_1^2 + 180 t_3 z + 10 t_1 z^3 + 32 z^6), \\
y_3^{\ast}&=&  -\frac{1}{64800m^5}
 c_0^5 (-200 t_1^5 - 12000 t_3^3 - 4200 t_1^3 t_3 z + 
         6050 t_1^4 z^3 - 
                3150 t_1^2 t_3 z^4\\
&& + 18745 t_1^3 z^6 + 3600 t_1 t_3 z^7 + 
         21040 t_1^2 z^9 + 3360 t_3 z^{10} \\
&&+ 
                10640 t_1 z^{12} + 2176 z^{15}).\\
\end{array}
\right.
\end{equation}

It is clear that (\ref{equation:def-y-t-b})
and the same
 by changing $y_j$ to $y_j^{\ast}\>(j=1,2,3)$ and putting $m=\frac{3c_0}{400}$.

In this manner, the formula (\ref{equation:def-y-t-a}) obtained by
Feigin, Valeri and Wright is proved by a group theoretic argument.

\section{The case  $H_4$}

This section contains the main result of this paper.

There are seven algebraic prepotentials denoted by
$H_4(j)\>(j=1,2,3,4,6,7,9)$ (cf. \cite{Se}) related with
the Coxeter group $W(H_4)$ of type $H_4$.
Feigin, Valeri and Wright \cite{FVW} determined 
a relation between the flat coordinates of the polynomial
prepotential $F_{H_4}$  and those
of each of algebraic prepotentials $F_{H_4(j)}\>(j=1,2,3,4,7)$.
Their idea is not applicable 
 to the remaining two cases $F_{H_4(6)},\>F_{H_4(9)}$
as mentioned in \cite{Wr}, p.118.
It is a problem to obtain
the relation between the flat coordinates of $F_{H_4}$ and 
those of $F_{H_{4}(6)}$ (or $F_{H_4(9)}$).
In this section,
we treat the case $H_4(9)$ and explain
how by our idea  developed in the previous section,
the flat coordinates of $F_{H_4}$ are expressed as rational functions
of the flat coordinates and an additional algebraic function $w_0$.

\subsection{A realization of $W(H_4)$ as a matrix group}

We first recall the definition of the Coxeter group $W(H_4)$ of type $H_4$.
It is known that $W(H_4)$ is generated by four elements $s_1,s_2,s_3,s_4$
with the relations:
$$
s_j^2=1\>(j=1,2,3,4),
(s_1s_2)^5=(s_1s_3)^2=(s_1s_4)^2=
(s_2s_3)^3=(s_2s_4)^2=(s_3s_4)^3=1.
$$
It is clear from the definition that the subgroup of $W(H_4)$
generated by $s_1,s_2,s_3$ is identified with $W(H_3)$.

We put $e_1=(1,0,0,0),e_2=(0,1,0,0),e_3=(0,0,1,0),e_4=(0,0,0,1)$ and denote by  $V$ the vector space spanned by  $e_1,e_2,e_3,e_4$ over ${\bf R}$.
We introduce four vectors
\begin{equation}
\label{equation:def-root-H4}
\alpha_1=e_1,\>\alpha_2=-\frac{1}{2}(ae_1+\bar{a}e_2+e_3),\>\alpha_3=e_3,
\alpha_4=-\frac{1}{2}(ae_2+e_3+\bar{a}e_4),
\end{equation}
where $a,\>\bar{a}$ are those defined in
(\ref{equation:def-a-bara}).
We denote by ${\tt u}=(u_1,u_2,u_3,u_4)$ the linear coordinate of $V$.
Let ${\tilde s}_j$ be the reflection of $V$ corresponding to $\alpha_j\>(j=1,2,3,4)$.
Then the matrix expression of $s_j\>(j=1,2,3,4)$ are given by ${\tilde s}_j\>(j=1,2,3,4)$,
in other words,
$$
\sigma(s_j)=\left(\begin{array}{cc}
\rho(s_j)&0\\
0&1\\
\end{array}\right)\>(j=1,2,3),
\quad
\sigma(s_4)=\frac{1}{2}\left(\begin{array}{cccc}
2&0&0&0\\
0&-{\bar a}&-a&1\\
0&-a&1&-{\bar a}\\
0&1&-{\bar a}&a\\
\end{array}\right).
$$

The group ${\cal W}(H_4)$  generated by  $\sigma(s_j)\>(j=1,2,3,4)$ defines a representation of $W(H_4)$.
In the sequel,
we identify the group
$W(H_4)$ with ${\cal W}(H_4)=\sigma(W(H_4))$.

As in the case $W(H_3)$, 
we define another representation of $W(H_4)$ which we are going to construct.
Let  $\sigma^{\ast}(s_j)=\sigma(s_j)|_{\sqrt{5}\to-\sqrt{5}}\>(j=1,2,3,4)$ be four matrices.
Then ${\cal W}^{\ast}(H_4)$ is the group generated by $\sigma^{\ast}(s_j)\>(j=1,2,3,4)$.
It is underlined here that ${\cal W}(H_4)$ and ${\cal W}^{\ast}(H_4)$ are isomorphic but not conjugate.

\subsection{Basic invariants of ${\cal W}(H_4)$}

In this subsection, we construct basic invariants of ${\cal W}(H_4)$.
For this purpose, we need some preparation.
At first we construct basic invariants of the reflection group of type $H_3$.
It is easy to see that the group $H$  generated by $\sigma(s_j)\>(j=1,2,3)$     is a subgroup of 
${\cal W}(H_4)$
 identified with ${\cal W}(H_3)$,
it follows from (\ref{equation:def-I-ep}) that
the polynomials $h_2,h_6,h_{10}$ given below are  invariants of $H$:
\begin{equation}
\label{equation:def-Ia-ep}
\begin{array}{lll}
h_2 (u_1,u_2,u_3)&=&\epsilon_1,\\
h_6 (u_1,u_2,u_3)&=&-11\epsilon_3+\epsilon_1\epsilon_2+\sqrt{5}\delta,\\
h_{10} (u_1,u_2,u_3)&=&95\epsilon_2\epsilon_3-32\epsilon_1^2\epsilon_3-5\epsilon_1\epsilon_2^2+
2\epsilon_1^3\epsilon_2+3\sqrt{5}\epsilon_2\delta.\\
\end{array}
\end{equation}
(We note that
$h_2=\frac{1}{2}I_1,\>h_6=\frac{1}{20}I_2,\>h_{10}=\frac{1}{80}I_3$.)
Invariants of ${\cal W}(H_4)$ can be written as polynomials of $h_2,h_6,h_{10}$ and $u_4$.
We take basic invariants of ${\cal W}(H_4)$ 
defined by the following manner:
\begin{equation}
\label{equation:H4-invariants}
\left\{
{\small
\begin{array}{lll}
Z_2 &=& h_2 + u_4^2, \\
 Z_{12} &=& -h_{10} h_2 + \frac{3}{2} h_6^2 +( 11 h_{10}  - 2 h_2^5) u_4^2 + 
   ( 6 h_2^4 - 
        33 h_2 h_6 )u_4^4 - (14 h_2^3 - 33 h_6) u_4^6 \\
        &&+ 6 h_2^2 u_4^8 - 
    2 h_2 u_4^{10}, \\
Z_{20}&=&    
   \frac{1}{2} (2 h_{10}^2 - 3 h_2 h_6^3) +\frac{1}{2} (8 h_{10} h_2^4 - 42 h_2^3 h_6^2 + 
    57 h_6^3) u_4^2\\
    && - 
   2 (38 h_{10} h_2^3 - 2 h_2^8 + 57 h_{10} h_6 - 15 h_2^5 h_6 - 84 h_2^2 h_6^2) 
  u_4^4\\
  && + 
   2 h_2 (148 h_{10} h_2 - 10 h_2^6 - 69 h_2^3 h_6 - 153 h_6^2) u_4^6\\
   && - 
 2 (232 h_{10} h_2 - 22 h_2^6 - 201 h_2^3 h_6 - 147 h_6^2) 
     u_4^8 + 2 (90 h_{10} - 22 h_2^5 - 201 h_2^2 h_6) u_4^{10}\\
     && + 
 2 h_2 (22 h_2^3 + 69 h_6) u_4^{12} - 10 (2 h_2^3 + 3 h_6) u_4^{14 }+ 
   4 h_2^2 u_4^{16},\\
Z_{30}&=&
\frac{1}{15} (20 h_{10}^3 - 45 h_{10} h_2 h_6^3 + 27 h_6^5) \\
&&+ (8 h_{10}^2 h_2^4 - 
    42 h_{10} h_2^3 h_6^2 - 87 h_{10} h_6^3 - 6 h_2^5 h_6^3 + 
        135 h_2^2 h_6^4) 
  u_4^2 \\
  &&+ (-152 h_{10}^2 h_2^3 + 16 h_{10} h_2^8 + 348 h_{10}^2 h_6 + 
    60 h_{10} h_2^5 h_6 + 408 h_{10} h_2^2 h_6^2 - 
        84 h_2^7 h_6^2 \\
        &&+ 84 h_2^4 h_6^3 - 909 h_2 h_6^4) u_4^4 + 
   \frac{2}{3} (1752 h_{10}^2 h_2^2 - 516 h_{10} h_2^7 + 16 h_2^{12} - 
    3258 h_{10} h_2^4 h_6 \\
    &&+ 180 h_2^9 h_6 - 2862 h_{10} h_2 h_6^2 + 
        1998 h_2^6 h_6^2 + 432 h_2^3 h_6^3 + 3591 h_6^4) u_4^6 \\
        &&- 
   2 (1616 h_{10}^2 h_2 - 1072 h_{10} h_2^6 + 40 h_2^{11} - 5454 h_{10} h_2^3 h_6 + 
    336 h_2^8 h_6 + 453 h_{10} h_6^2 \\
    &&+ 3462 h_2^5 h_6^2 + 
        978 h_2^2 h_6^3) 
  u_4^8 +\frac{2}{5} (8100 h_{10}^2 - 17900 h_{10} h_2^5 + 672 h_2^{10} \\
  &&- 
    63330 h_{10} h_2^2 h_6 + 4020 h_2^7 h_6 + 
        49920 h_2^4 h_6^2 + 18375 h_2 h_6^3) 
  u_4^{10} \\
  &&+\frac{4}{3} (10662 h_{10} h_2^4 - 332 h_2^9 + 17631 h_{10} h_2 h_6 - 
    936 h_2^6 h_6 - 
        20547 h_2^3 h_6^2 - 4347 h_6^3) u_4^{12} \\
        &&- 
 4 (4214 h_{10} h_2^3 - 68 h_2^8 + 1905 h_{10} h_6 - 4647 h_2^2 h_6^2) u_4^{14}\\
 && + 
   4 h_2 (2506 h_{10} h_2 + 68 h_2^6 + 312 h_2^3 h_6 - 1407 h_6^2) u_4^{16}\\
   && - 
  \frac{8}{3} (1080 h_{10} h_2 + 166 h_2^6 + 603 h_2^3 h_6 - 405 h_6^2) 
  u_4^{18} \\
  &&+\frac{24}{5} (75 h_{10} + 56 h_2^5 + 140 h_2^2 h_6) u_4^{20} - 
   40 h_2 (2 h_2^3 + 3 h_6) u_4^{22} + \frac{32}{3} h_2^3 u_4^{24}.\\
       \end{array}
       }
\right.
\end{equation}

\begin{remark}
The invariants $Z_j\>(j=2,12,20,30)$ are the same as $y_4,y_3,y_2,y_1$ given in \cite{FVW}.
The difference is the change of $u_1$ with $u_4$,
$h_j(u_1,u_2,u_3)$ with $h_j(u_4,u_3,u_2)$ $(j=2,6,10)$.
\end{remark}

Concrete forms of basic ${\cal W}^{\ast}(H_4)$-invariants are expressed as follows.
Let $k_2,k_6,k_{10}$ be the basic invariants of the group ${\cal W}^{\ast}(H_3)$ defined by
$$
k_2=\frac{1}{2}J_1,\>k_6=\frac{1}{20}J_2,\>k_{10}=\frac{1}{80}J_3,
$$
where $J_1,J_2,J_3$ are polynomials  given by
(\ref{equation:def-J-ep}).
By the substitution of  $h_2,h_6,h_{10}$
by $k_2,k_6,k_{10}$,
we obtain polynomial $Z_j^{\ast}$ from $Z_j$ $(j=2,12,20,30)$
(cf. (\ref{equation:H4-invariants})),
namely,
$$
\begin{array}{lll}
Z_2^{\ast} &=& k_2 + u_4^2, \\
 Z_{12}^{\ast} &=& -k_{10} k_2 + \frac{3}{2} k_6^2 +( 11 k_{10}  - 2 k_2^5) u_4^2 + 
   ( 6 k_2^4 - 
        33 k_2 k_6 )u_4^4 - (14 k_2^3 - 33 k_6) u_4^6 \\
        &&+ 6 k_2^2 u_4^8 - 
    2 k_2 u_4^{10}, \\
    {\rm etc}.&&\\
\end{array}
$$
Then it follows from the definition that $Z_2^{\ast},Z_{12}^{\ast},Z_{20}^{\ast},Z_{30}^{\ast}$
form the basic invariants of ${\cal W}^{\ast}(H_4)$.

\subsection{The discriminant of  ${\cal W}(H_4)$}

We define the set ${\cal R}$ of  60  linear forms defining reflections of ${\cal W}(H_4)$, that is,
${\cal R}$ is the set of the following linear forms:
\begin{equation}
\label{equation:60-linear forms}
\begin{array}{l}
u_1, u_2, u_3, u_4, \\
u_1 + u_2 + u_3 + u_4, u_1 - u_2 + u_3 + u_4, 
 u_1 + u_2 - u_3 + u_4, u_1 + u_2 + u_3 - u_4, \\
 u_1 - u_2 - u_3 + u_4, 
 u_1 - u_2 + u_3 - u_4, u_1 + u_2 - u_3 - u_4, 
   u_1 - u_2 - u_3 - u_4,\\
u_1 - a u_3 + \bar{a} u_4, 
 u_1 + a u_3 - \bar{a} u_4, 
 u_1 - a u_3 - \bar{a} u_4, 
   u_1 + a u_3 + \bar{a} u_4, \\
 u_1 + \bar{a} u_2 - a u_4, 
 u_1 - \bar{a} u_2 + a u_4, 
   u_1 - \bar{a} u_2 - a u_4, 
 u_1 + \bar{a} u_2 + a u_4, \\
 u_1 - a u_2 + \bar{a} u_3, 
   u_1 + a u_2 - \bar{a} u_3, 
 u_1 - a u_2 - \bar{a} u_3, 
 u_1 + a u_2 + \bar{a} u_3, \\
   \bar{a} u_1 + 
  u_2 - a u_3, -\bar{a} u_1 + 
  u_2 - a u_3, -\bar{a} u_1 + 
  u_2 + a u_3, 
   \bar{a} u_1 + 
  u_2 + a u_3, \\
  -a u_1 + 
  u_2 + \bar{a} u_4, a u_1 + 
  u_2 - \bar{a} u_4, 
   -a u_1 + 
  u_2 - \bar{a} u_4, a u_1 + 
  u_2 + \bar{a} u_4, \\
 u_2 + \bar{a} u_3 - a u_4, 
   u_2 - \bar{a} u_3 + a u_4, 
 u_2 - \bar{a} u_3 - a u_4, 
 u_2 + \bar{a} u_3 + a u_4, \\
   -a u_1 + \bar{a} u_2 + 
  u_3, -a u_1 - \bar{a} u_2 + 
  u_3, a u_1 - \bar{a} u_2 + u_3, 
   a u_1 + \bar{a} u_2 + 
  u_3, \\
  -a u_2 + 
  u_3 + \bar{a} u_4, a u_2 + 
  u_3 - \bar{a} u_4, 
   -a u_2 + 
  u_3 - \bar{a} u_4, a u_2 + 
  u_3 + \bar{a} u_4,\\
   \bar{a} u_1 + 
  u_3 - a u_4, 
   -\bar{a} u_1 + 
  u_3 - a u_4, -\bar{a} u_1 + 
  u_3 + a u_4, \bar{a} u_1 + 
  u_3 + a u_4, \\
   \bar{a} u_1 - a u_2 + 
  u_4, -\bar{a} u_1 - a u_2 + 
  u_4, -\bar{a} u_1 + a u_2 + u_4, 
   \bar{a} u_1 + a u_2 + 
  u_4, \\
  -a u_1 + \bar{a} u_3 + 
  u_4, -a u_1 - \bar{a} u_3 + u_4, 
   a u_1 - \bar{a} u_3 + 
  u_4, a u_1 + \bar{a} u_3 + 
  u_4, \\
  \bar{a} u_2 - a u_3 + u_4, 
   -\bar{a} u_2 - a u_3 + 
  u_4, -\bar{a} u_2 + a u_3 + 
  u_4, \bar{a} u_2 + a u_3 + u_4.\\
  \end{array}
 \end{equation}
 Let $D(u)=\prod_{\varphi\in{\cal R}}\varphi$.
 Then $D(u)^2$ is ${\cal W}(H_4)$-invariant.
 As a consequence,
 it is a polynomial of $Z_2,Z_{12},Z_{20},Z_{30}$. 
 In fact, it can be shown that
 $D(u)^2$ coincides with ${\tilde D}(Z_2,Z_{12},Z_{20},Z_{30})$
 up to a constant factor, where
 {\small
 \begin{equation}
 \label{equation:def-DFZ}
 \begin{array}{ll}
 &{\tilde D}(Z_2,Z_{12},Z_{20},Z_{30})\\
 =&64 (3888 Z_{12}^{10 }+ 250 Z_{12}^9 Z_2^6 - 35640 Z_{12}^8 Z_2^2 Z_{20} - 
    2250 Z_{12}^7 Z_2^8 Z_{20} + 110655 Z_{12}^6 Z_2^4 Z_{20}^2\\
&     + 
    6750 Z_{12}^5 Z_2^{10} Z_{20}^2 
    - 32400 Z_{12}^5 Z_{20}^3 - 
        123255 Z_{12}^4 Z_2^6 Z_{20}^3 - 6750 Z_{12}^3 Z_2^{12} Z_{20}^3\\
  &       - 
    256500 Z_{12}^3 Z_2^2 Z_{20}^4 - 2835 Z_{12}^2 Z_2^8 Z_{20}^4 + 
    40500 Z_{12} Z_2^4 Z_{20}^5 + 2187 Z_2^{10} Z_{20}^5 + 67500 Z_{20}^6)\\
&  -7200 Z_2 (156 Z_{12}^7 Z_2^2 + 10 Z_{12}^6 Z_2^8 - 963 Z_{12}^5 Z_2^4 Z_{20} - 
    60 Z_{12}^4 Z_2^{10} Z_{20} + 2160 Z_{12}^4 Z_{20}^2\\
 &    + 1698 Z_{12}^3 Z_2^6 Z_{20}^2 + 
    90 Z_{12}^2 Z_2^{12} Z_{20}^2 + 
        2970 Z_{12}^2 Z_2^2 Z_{20}^3 + 9 Z_{12} Z_2^8 Z_{20}^3 - 
    18 Z_2^4 Z_{20}^4)Z_{30}\\
& -10800 (108 Z_{12}^5 - 150 Z_{12}^4 Z_2^6 - 
    10 Z_{12}^3 Z_2^{12} + 1305 Z_{12}^3 Z_2^2 Z_{20} + 585 Z_{12}^2 Z_2^8 Z_{20}\\
 &    + 
        30 Z_{12} Z_2^{14} Z_{20} + 1215 Z_{12} Z_2^4 Z_{20}^2 + 15 Z_2^{10} Z_{20}^2 + 
    450 Z_{20}^3)Z_{30}^2\\
&    -13500 
  Z_2^3 (405 Z_{12}^2 + 90 Z_{12} Z_2^6 + 4 Z_2^{12} + 135 Z_2^2 Z_{20})Z_{30}^3\\
&+   1366875Z_{30}^4.\\
  \end{array}
 \end{equation}
  }

\subsection{Four dimensional representation of $W(H_4)$ on the space of polynomials of $7^{{\rm th}}$ degree}

We first introduce a  polynomial of degree 7 by
$$
f(u_1,u_2,u_3,u_4)=\frac{1}{168}u_4(-21h_6+14h_2^2u_4^2-14h_2u_4^4+2u_4^6)
$$
and using $f(u_1,u_2,u_3,u_4)$, define
$$
\begin{array}{lll}
P_1(u_1,u_2,u_3,u_4)&=&f(u_4,u_3,u_2,u_1),\\
P_2(u_1,u_2,u_3,u_4)&=&f(u_3,u_4,u_1,u_2),\\
P_3(u_1,u_2,u_3,u_4)&=&f(u_2,u_1,u_4,u_3),\\
P_4(u_1,u_2,u_3,u_4)&=&f(u_1,u_2,u_3,u_4).\\
\end{array}
$$
We  also define 
$$
v_k=P_k({\tt u})\>(k=1,2,3,4)
$$
and 
$${\tt v}=(v_1,v_2,v_3,v_4).
$$

The  equality in the lemma below plays a basic role in the subsequent consideration:

\begin{lemma}
\label{lemma:lemma-H4-H4}
$$
(P_1({\tt u}\sigma(s_j)),P_2({\tt u}\sigma(s_j)),P_3({\tt u}\sigma(s_j)),
P_4({\tt u}\sigma(s_j)))=
{\tt v}\sigma^{\ast}(s_j)\quad(j=1,2,3,4).
$$
\end{lemma}

Then $Y_j=Z_j^{\ast}(P_1,P_2,P_3,P_4)\>(j=2,12,20,30)$ are the basic invariants of ${\cal W}^{\ast}(H_4)$
if  $(P_1,P_2,P_3,P_4)$ is regarded as a linear coordinate.
Further more, it follows from Lemma \ref{lemma:lemma-H4-H4} that
 $Z_j^{\ast}(P_1,P_2,P_3,P_4)\>(j=2,12,20,30)$ are  invariants of ${\cal W}(H_4)$.
This means that each 
 $Y_j$ is expressed as a polynomial of $Z_j\>(j=2,12,20,30)$.

\begin{theorem}
\label{theorem:formula-Y-Z}
We put  $Y_j=Z_j^{\ast}(P_1,P_2,P_3,P_4)\>(j=2,12,20,30)$ as above.
Then the  concrete forms of 
 $Y_2,Y_{12},Y_{20},Y_{30}$ as  polynomials of $Z_j\>(j=2,12,20,30)$
 are given as follows:

{\footnotesize
\begin{verbatim}
Y2=
1/(2^5*3^2*7^2)*(+3*7 * Z12*Z2+2 * Z2^7)

Y12=
1/(2^32*3^14*5^2*7^12)*(
+2^3*3^12*5^4*7 * Z12^2*Z30^2
+2^3*3^12*5^4*7 * Z20*Z2^2*Z30^2
+3^14*5^5*7 * Z20^2*Z12*Z2*Z30
+2*3^10*5^7*7 * Z20^3*Z12^2
+2*3^10*5^7*7 * Z20^4*Z2^2
-2^8*3^8*5^5 * Z12^7
-2^4*3^9*5^3*7*103 * Z12^4*Z2^3*Z30
-2^2*3^8*5^4*7*113 * Z20*Z12^5*Z2^2
+2^6*3^11*5^3*7 * Z12*Z2^6*Z30^2
+2^3*3^9*5^2*7*41*43 * Z20*Z12^2*Z2^5*Z30
+2^3*3^8*5^4*7*17 * Z20^2*Z12^3*Z2^4
-2*3^9*5^2*7*17*29 * Z20^2*Z2^7*Z30
+2^2*3^9*5^3*7^2*137 * Z20^3*Z12*Z2^6
-2^4*3^6*5^2*7*1861 * Z12^6*Z2^6
-2^5*3^7*5^2*7*11*17 * Z12^3*Z2^9*Z30
-2^3*3^7*5*7*11*41*53 * Z20*Z12^4*Z2^8
+2^5*3^9*5^3*7 * Z2^12*Z30^2
+2^3*3^8*5*7*307 * Z20*Z12*Z2^11*Z30
+2^3*3^8*7*11*6089 * Z20^2*Z12^2*Z2^10
+2^5*3^6*7*97*307 * Z20^3*Z2^12
-2^4*3^5*7^3*13*277 * Z12^5*Z2^12
-2^5*3^6*5*7*17*23 * Z12^2*Z2^15*Z30
+2^5*3^5*5^2*7*13*139 * Z20*Z12^3*Z2^14
-2^4*3^6*5*7*17*19 * Z20*Z2^17*Z30
+2^4*3^7*5*7*3011 * Z20^2*Z12*Z2^16
-2^5*3^2*5^2*7^2*37*71 * Z12^4*Z2^18
-2^4*3^5*5*7*13*17 * Z12*Z2^21*Z30
+2^5*3^2*5*7^2*23*541 * Z20*Z12^2*Z2^20
+2^6*3^2*5*7*47*137 * Z20^2*Z2^22
-2^8*5^2*7*11*127 * Z12^3*Z2^24
-2^5*3*5*7*13*17 * Z2^27*Z30
+2^7*3*5*7*4793 * Z20*Z12*Z2^26
-2^7*3^2*5^2*7^2 * Z12^2*Z2^30
+2^7*3^2*5^2*7^2 * Z20*Z2^32
)

Y20=
1/(2^54*3^22*5^4*7^20)*(
-2^5*3^20*5^8*7 * Z20*Z30^4
+3^18*5^11*7*11 * Z20^4*Z30^2
-2^4*3^16*5^13 * Z20^7
-2^6*3^17*5^7*7*13 * Z12^4*Z2*Z30^3
+2^8*3^16*5^9*7 * Z20*Z12^5*Z30^2
-2^7*3^18*5^8*7 * Z12*Z2^4*Z30^4
+2^7*3^17*5^7*7*107 * Z20*Z12^2*Z2^3*Z30^3
+2^9*3^16*5^10*7 * Z20^2*Z12^3*Z2^2*Z30^2
+2^3*3^15*5^10*7*271 * Z20^3*Z12^4*Z2*Z30
+2^3*3^14*5^11*7*13 * Z20^4*Z12^5
+2^6*3^17*5^6*7*467 * Z20^2*Z2^5*Z30^3
+2^4*3^16*5^10*7*71 * Z20^3*Z12*Z2^4*Z30^2
+2^5*3^15*5^10*7*53 * Z20^4*Z12^2*Z2^3*Z30
+2^4*3^14*5^12*7*13 * Z20^5*Z12^3*Z2^2
-2^3*3^15*5^9*7*389 * Z20^5*Z2^5*Z30
+2^11*3^13*5^8*7*11 * Z12^9*Z2*Z30
-2^3*3^17*5^12*7 * Z20^6*Z12*Z2^4
+2^13*3^12*5^9*7 * Z20*Z12^10
+2^12*3^14*5^7*7^2 * Z12^6*Z2^4*Z30^2
+2^4*3^14*5^7*7*859 * Z20*Z12^7*Z2^3*Z30
+2^5*3^13*5^10*7*103 * Z20^2*Z12^8*Z2^2
+2^7*3^15*5^8*7*83 * Z12^3*Z2^7*Z30^3
+2^6*3^16*5^5*7^3*59 * Z20*Z12^4*Z2^6*Z30^2
+2^4*3^14*5^10*7^3 * Z20^2*Z12^5*Z2^5*Z30
-2^5*3^12*5^8*7^2*2917 * Z20^3*Z12^6*Z2^4
-2^7*3^17*5^7*7 * Z2^10*Z30^4
+2^6*3^15*5^6*7*41*101 * Z20*Z12*Z2^9*Z30^3
+2^6*3^15*5^5*7*241*421 * Z20^2*Z12^2*Z2^8*Z30^2
+2^6*3^13*5^7*7*23*1031 * Z20^3*Z12^3*Z2^7*Z30
+2^5*3^14*5^8*7^2*479 * Z20^4*Z12^4*Z2^6
+2^4*3^14*5^4*7*77983 * Z20^3*Z2^10*Z30^2
-2^5*3^15*5^8*7*11*89 * Z20^4*Z12*Z2^9*Z30
-2^5*3^14*5^8*7*3917 * Z20^5*Z12^2*Z2^8
-2^5*3^12*5^7*7^2*11^2 * Z12^11*Z2^4
+2^6*3^13*5^5*7*23*2953 * Z12^8*Z2^7*Z30
+2^6*3^13*5^7*7^2*79 * Z20^6*Z2^10
+2^10*3^11*5^6*7^2*23*89 * Z20*Z12^9*Z2^6
-2^7*3^16*5^5*7^3*41 * Z12^5*Z2^10*Z30^2
-2^5*3^13*5^5*7^3*37*109 * Z20*Z12^6*Z2^9*Z30
-2^5*3^12*5^5*7*901679 * Z20^2*Z12^7*Z2^8
+2^9*3^15*5^6*7^2*23 * Z12^2*Z2^13*Z30^3
+2^8*3^14*5^5*7^2*13*127 * Z20*Z12^3*Z2^12*Z30^2
+2^6*3^13*5^3*7^2*171827 * Z20^2*Z12^4*Z2^11*Z30
+2^6*3^16*5^5*7^2*277 * Z20^3*Z12^5*Z2^10
+2^7*3^14*5^5*7^2*241 * Z20*Z2^15*Z30^3
+2^6*3^18*5^4*7^2*103 * Z20^2*Z12*Z2^14*Z30^2
-2^6*3^14*5^3*7^2*11*13*29*101 * Z20^3*Z12^2*Z2^13*Z30
-2^7*3^12*5^5*7^2*137*599 * Z20^4*Z12^3*Z2^12
-2^5*3^13*5^2*7^3*111491 * Z20^4*Z2^15*Z30
+2^7*3^13*5^6*7^2*523 * Z20^5*Z12*Z2^14
+2^6*3^13*5^4*7*42589 * Z12^10*Z2^10
+2^6*3^10*5^3*7*17*2056963 * Z12^7*Z2^13*Z30
+2^11*3^9*5^3*7*11*295909 * Z20*Z12^8*Z2^12
-2^8*3^11*5^4*7^2*13*19*211 * Z12^4*Z2^16*Z30^2
-2^6*3^10*5^3*7^2*11*318431 * Z20*Z12^5*Z2^15*Z30
-2^6*3^9*5^3*7^2*191*38569 * Z20^2*Z12^6*Z2^14
+2^8*3^11*5^5*7^2*1621 * Z12*Z2^19*Z30^3
+2^9*3^10*5^5*7^2*11*233 * Z20*Z12^2*Z2^18*Z30^2
-2^7*3^10*5^3*7^2*5484503 * Z20^2*Z12^3*Z2^17*Z30
-2^11*3^10*5*7^2*11448077 * Z20^3*Z12^4*Z2^16
-2^5*3^10*5^3*7^2*163*1367 * Z20^2*Z2^20*Z30^2
-2^8*3^9*5^2*7^2*12136211 * Z20^3*Z12*Z2^19*Z30
+2^7*3^8*5*7^2*409*701*811 * Z20^4*Z12^2*Z2^18
+2^7*3^8*7^2*283*759631 * Z20^5*Z2^20
+2^7*3^10*5^2*7^2*1480541 * Z12^9*Z2^16
+2^7*3^7*5^2*7^3*14996291 * Z12^6*Z2^19*Z30
+2^8*3^6*5*7^2*798329969 * Z20*Z12^7*Z2^18
-2^14*3^8*5^4*7^2*3461 * Z12^3*Z2^22*Z30^2
-2^8*3^7*5^2*7^2*661*93463 * Z20*Z12^4*Z2^21*Z30
-2^7*3^8*5*7^3*751*81553 * Z20^2*Z12^5*Z2^20
+2^9*3^10*5^3*7^2*1621 * Z2^25*Z30^3
-2^8*3^8*5^3*7^2*13*13597 * Z20*Z12*Z2^24*Z30^2
-2^12*3^9*5^2*7^2*601^2 * Z20^2*Z12^2*Z2^23*Z30
-2^13*3^6*5*7^2*701*3121 * Z20^3*Z12^3*Z2^22
-2^7*3^7*5^2*7^2*17*241877 * Z20^3*Z2^25*Z30
+2^8*3^7*5*7^2*95436673 * Z20^4*Z12*Z2^24
+2^8*3^6*5^2*7^3*4115879 * Z12^8*Z2^22
+2^10*3^7*5*7^2*8375963 * Z12^5*Z2^25*Z30
-2^13*3^5*5*7^2*131*233*331 * Z20*Z12^6*Z2^24
-2^10*3^9*5^3*7^2*11*743 * Z12^2*Z2^28*Z30^2
-2^8*3^6*5^2*7^2*28005389 * Z20*Z12^3*Z2^27*Z30
-2^8*3^7*5^3*7^3*402503 * Z20^2*Z12^4*Z2^26
-2^8*3^7*5^3*7^2*11*29*79 * Z20*Z2^30*Z30^2
-2^8*3^7*5^2*7^2*199*10889 * Z20^2*Z12*Z2^29*Z30
+2^9*3^6*5^3*7^2*421*5749 * Z20^3*Z12^2*Z2^28
+2^10*3^5*5^2*7^2*11*123191 * Z20^4*Z2^30
+2^9*3^6*5^2*7^3*17*139*157 * Z12^7*Z2^28
+2^9*3^4*5^3*7^2*41*25867 * Z12^4*Z2^31*Z30
-2^12*3^3*5^2*7^3*13*162263 * Z20*Z12^5*Z2^30
-2^10*3^5*5^3*7^2*71*409 * Z12*Z2^34*Z30^2
-2^9*3^4*5^2*7^2*19*127*3109 * Z20*Z12^2*Z2^33*Z30
+2^9*3^3*5^2*7^2*607*32801 * Z20^2*Z12^3*Z2^32
-2^12*3^4*5^2*7^2*173*241 * Z20^2*Z2^35*Z30
+2^10*3^5*5^2*7^2*11*118493 * Z20^3*Z12*Z2^34
+2^10*3^5*5^4*7^2*17*751 * Z12^6*Z2^34
+2^10*3*5^4*7^2*13^2*439 * Z12^3*Z2^37*Z30
-2^11*5^3*7^2*17*901141 * Z20*Z12^4*Z2^36
-2^8*3^4*5^2*7^2*71*409 * Z2^40*Z30^2
-2^10*3*5^2*7^2*11*218947 * Z20*Z12*Z2^39*Z30
+2^10*3*5^2*7^2*53*430499 * Z20^2*Z12^2*Z2^38
+2^12*5^3*7^3*11*4021 * Z20^3*Z2^40
+2^12*5^5*7^3*853 * Z12^5*Z2^40
+2^11*3*5^3*7^3*13*17 * Z12^2*Z2^43*Z30
-2^14*5^3*7^3*10867 * Z20*Z12^3*Z2^42
-2^11*3*5^3*7^3*13*17 * Z20*Z2^45*Z30
+2^12*3*5^3*7^3*11^2*61 * Z20^2*Z12*Z2^44
+2^12*3^2*5^4*7^4 * Z12^4*Z2^46
-2^13*3^2*5^4*7^4 * Z20*Z12^2*Z2^48
+2^12*3^2*5^4*7^4 * Z20^2*Z2^50
)

Y30=
1/(2^79*3^34*5^6*7^30)*(
-2^7*3^31*5^12 * Z30^7
+2^4*3^30*5^12*7*107 * Z20^3*Z30^5
-3^28*5^16*7*71 * Z20^6*Z30^3
+2^9*3^28*5^10*7*23 * Z12^5*Z30^5
+2^4*3^25*5^19*7 * Z20^9*Z30
+2^10*3^29*5^12*7 * Z12^2*Z2^3*Z30^6
+2^10*3^28*5^11*7*23 * Z20*Z12^3*Z2^2*Z30^5
+2^5*3^27*5^12*7*1129 * Z20^2*Z12^4*Z2*Z30^4
-2^6*3^26*5^14*7*41 * Z20^3*Z12^5*Z30^3
+2^11*3^29*5^11*7 * Z20*Z2^5*Z30^6
+2^6*3^28*5^11*7*1789 * Z20^2*Z12*Z2^4*Z30^5
-2^6*3^27*5^12*7*2081 * Z20^3*Z12^2*Z2^3*Z30^4
-2^7*3^26*5^15*7*41 * Z20^4*Z12^3*Z2^2*Z30^3
-2^2*3^25*5^15*7*5393 * Z20^5*Z12^4*Z2*Z30^2
-2^2*3^24*5^18*7*11 * Z20^6*Z12^5*Z30
-2^5*3^27*5^11*7^3*239 * Z20^4*Z2^5*Z30^4
-2^3*3^26*5^15*7^2*199 * Z20^5*Z12*Z2^4*Z30^3
-2^15*3^24*5^12*7 * Z12^10*Z30^3
-2^4*3^25*5^15*7^2*157 * Z20^6*Z12^2*Z2^3*Z30^2
-2^3*3^24*5^19*7*11 * Z20^7*Z12^3*Z2^2*Z30
+2^4*3^23*5^18*7*23 * Z20^8*Z12^4*Z2
+2^7*3^25*5^10*7*1039 * Z12^7*Z2^3*Z30^4
+2^7*3^25*5^11*7*71*127 * Z20*Z12^8*Z2^2*Z30^3
+2^2*3^25*5^14*7*83*89 * Z20^7*Z2^5*Z30^2
+2^9*3^23*5^13*7*13*73 * Z20^2*Z12^9*Z2*Z30^2
+2^2*3^24*5^19*7*37 * Z20^8*Z12*Z2^4*Z30
+2^11*3^22*5^15*7*23 * Z20^3*Z12^10*Z30
+2^5*3^24*5^18*7 * Z20^9*Z12^2*Z2^3
-2^9*3^26*5^10*7^2*373 * Z12^4*Z2^6*Z30^5
-2^7*3^26*5^9*7^2*59*163 * Z20*Z12^5*Z2^5*Z30^4
-2^7*3^26*5^11*7^2*1123 * Z20^2*Z12^6*Z2^4*Z30^3
+2^3*3^23*5^12*7*11*263*349 * Z20^3*Z12^7*Z2^3*Z30^2
+2^3*3^23*5^14*7*22811 * Z20^4*Z12^8*Z2^2*Z30
-2^4*3^24*5^17*7*19 * Z20^10*Z2^5
+2^6*3^21*5^16*7*17*37 * Z20^5*Z12^9*Z2
+2^9*3^28*5^11*7^2 * Z12*Z2^9*Z30^6
+2^10*3^27*5^10*7^2*59 * Z20*Z12^2*Z2^8*Z30^5
-2^6*3^26*5^10*7^2*36697 * Z20^2*Z12^3*Z2^7*Z30^4
-2^5*3^24*5^10*7^2*491*1493 * Z20^3*Z12^4*Z2^6*Z30^3
-2^3*3^24*5^12*7^2*24977 * Z20^4*Z12^5*Z2^5*Z30^2
+2^4*3^24*5^13*7^2*11*17*107 * Z20^5*Z12^6*Z2^4*Z30
+2^19*3^19*5^13*7 * Z12^15*Z30
+2^4*3^21*5^15*7^2*2053 * Z20^6*Z12^7*Z2^3
+2^6*3^26*5^9*7^2*23*31 * Z20^2*Z2^10*Z30^5
-2^5*3^25*5^12*7^2*19*307 * Z20^3*Z12*Z2^9*Z30^4
-2^6*3^25*5^10*7^2*59*4219 * Z20^4*Z12^2*Z2^8*Z30^3
-2^5*3^24*5^12*7^4*373 * Z20^5*Z12^3*Z2^7*Z30^2
+2^10*3^20*5^11*7^2*7517 * Z12^12*Z2^3*Z30^2
+2^10*3^22*5^13*7^2*197 * Z20^6*Z12^4*Z2^6*Z30
+2^12*3^20*5^12*7^2*677 * Z20*Z12^13*Z2^2*Z30
+2^4*3^24*5^13*7^2*19087 * Z20^7*Z12^5*Z2^5
+2^14*3^19*5^14*7*107 * Z20^2*Z12^14*Z2
+2^3*3^24*5^9*7^2*805289 * Z20^5*Z2^10*Z30^3
-2^8*3^22*5^9*7^2*1559 * Z12^9*Z2^6*Z30^3
+2^4*3^23*5^13*7^2*46549 * Z20^6*Z12*Z2^9*Z30^2
-2^9*3^22*5^10*7^2*31*1049 * Z20*Z12^10*Z2^5*Z30^2
+2^4*3^24*5^13*7*290527 * Z20^7*Z12^2*Z2^8*Z30
-2^4*3^20*5^11*7^2*17*29^2*31 * Z20^2*Z12^11*Z2^4*Z30
+2^4*3^23*5^15*7*12569 * Z20^8*Z12^3*Z2^7
-2^13*3^18*5^13*7^2*683 * Z20^3*Z12^12*Z2^3
+2^8*3^22*5^9*7^5*977 * Z12^6*Z2^9*Z30^4
+2^8*3^23*5^8*7^2*139*8543 * Z20*Z12^7*Z2^8*Z30^3
+2^5*3^22*5^9*7^2*509*32057 * Z20^2*Z12^8*Z2^7*Z30^2
+2^3*3^24*5^12*7*31*3581 * Z20^8*Z2^10*Z30
+2^5*3^19*5^11*7^2*521*4133 * Z20^3*Z12^9*Z2^6*Z30
-2^4*3^22*5^16*7*1453 * Z20^9*Z12*Z2^9
+2^5*3^19*5^11*7^2*31*193*983 * Z20^4*Z12^10*Z2^5
-2^11*3^23*5^11*7^3*17 * Z12^3*Z2^12*Z30^5
-2^9*3^23*5^8*7^4*11*13^3 * Z20*Z12^4*Z2^11*Z30^4
+2^6*3^22*5^7*7^3*9993337 * Z20^2*Z12^5*Z2^10*Z30^3
+2^4*3^21*5^9*7^3*44112559 * Z20^3*Z12^6*Z2^9*Z30^2
+2^4*3^21*5^10*7^2*7079*11617 * Z20^4*Z12^7*Z2^8*Z30
-2^5*3^22*5^12*7^2*41851 * Z20^5*Z12^8*Z2^7
-2^15*3^18*5^12*7*73 * Z12^17*Z2^3
+2^10*3^26*5^10*7^2 * Z2^15*Z30^6
+2^9*3^29*5^9*7^2 * Z20*Z12*Z2^14*Z30^5
-2^10*3^23*5^8*7^3*31*3469 * Z20^2*Z12^2*Z2^13*Z30^4
-2^7*3^21*5^8*7^3*173*71471 * Z20^3*Z12^3*Z2^12*Z30^3
+2^12*3^22*5^8*7^3*23*9043 * Z20^4*Z12^4*Z2^11*Z30^2
+2^5*3^22*5^10*7^4*271771 * Z20^5*Z12^5*Z2^10*Z30
+2^5*3^19*5^10*7*8052413 * Z12^14*Z2^6*Z30
+2^6*3^20*5^11*7^3*11*31*4391 * Z20^6*Z12^6*Z2^9
+2^8*3^18*5^11*7*19*193*211 * Z20*Z12^15*Z2^5
-2^6*3^22*5^7*7^2*11*1070429 * Z20^3*Z2^15*Z30^4
+2^5*3^22*5^8*7^2*17*83*10627 * Z20^4*Z12*Z2^14*Z30^3
+2^5*3^23*5^8*7^3*15042373 * Z20^5*Z12^2*Z2^13*Z30^2
+2^8*3^21*5^8*7^2*11*149*457 * Z12^11*Z2^9*Z30^2
+2^11*3^21*5^10*7^4*11*2539 * Z20^6*Z12^3*Z2^12*Z30
+2^6*3^24*5^9*7^2*67*2027 * Z20*Z12^12*Z2^8*Z30
+2^7*3^22*5^11*7^2*43*19681 * Z20^7*Z12^4*Z2^11
+2^6*3^18*5^10*7^2*6673981 * Z20^2*Z12^13*Z2^7
-2^8*3^21*5^7*7^2*13*17*113*389 * Z12^8*Z2^12*Z30^3
+2^4*3^22*5^7*7^2*383*208697 * Z20^6*Z2^15*Z30^2
-2^12*3^20*5^7*7^2*53*197*421 * Z20*Z12^9*Z2^11*Z30^2
+2^5*3^23*5^11*7*89*13921 * Z20^7*Z12*Z2^14*Z30
-2^8*3^19*5^9*7^2*11*1327889 * Z20^2*Z12^10*Z2^10*Z30
-2^5*3^22*5^11*7^2*13*43*47 * Z20^8*Z12^2*Z2^13
+2^6*3^20*5^9*7^2*61*151*1609 * Z20^3*Z12^11*Z2^9
-2^16*3^24*5^8*7^2*43 * Z12^5*Z2^15*Z30^4
+2^8*3^25*5^7*7^2*1150987 * Z20*Z12^6*Z2^14*Z30^3
+2^5*3^20*5^6*7^2*523*56985457 * Z20^2*Z12^7*Z2^13*Z30^2
+2^6*3^19*5^7*7^2*13*427218427 * Z20^3*Z12^8*Z2^12*Z30
-2^5*3^22*5^10*7*257*5003 * Z20^9*Z2^15
-2^6*3^19*5^9*7^2*127*1214933 * Z20^4*Z12^9*Z2^11
-2^9*3^22*5^9*7^2*19*241 * Z12^2*Z2^18*Z30^5
-2^9*3^22*5^8*7^2*1432103 * Z20*Z12^3*Z2^17*Z30^4
+2^7*3^21*5^6*7^2*37*47*272369 * Z20^2*Z12^4*Z2^16*Z30^3
+2^5*3^20*5^5*7^2*229*28477*38953 * Z20^3*Z12^5*Z2^15*Z30^2
+2^5*3^19*5^7*7^2*281*109191547 * Z20^4*Z12^6*Z2^14*Z30
+2^6*3^18*5^8*7^2*11*269192369 * Z20^5*Z12^7*Z2^13
-2^7*3^18*5^10*7^2*29989 * Z12^16*Z2^9
-2^8*3^22*5^8*7^2*3361 * Z20*Z2^20*Z30^5
-2^7*3^21*5^7*7^2*24413693 * Z20^2*Z12*Z2^19*Z30^4
+2^11*3^20*5^6*7^2*257*257903 * Z20^3*Z12^2*Z2^18*Z30^3
+2^6*3^20*5^7*7^2*101*1129*50101 * Z20^4*Z12^3*Z2^17*Z30^2
+2^6*3^19*5^6*7^2*389*631*747853 * Z20^5*Z12^4*Z2^16*Z30
+2^7*3^17*5^8*7^3*11*2122963 * Z12^13*Z2^12*Z30
+2^8*3^19*5^8*7^2*223*1931*2789 * Z20^6*Z12^5*Z2^15
+2^7*3^18*5^8*7^2*239*130099 * Z20*Z12^14*Z2^11
+2^4*3^20*5^5*7^2*12569*345689 * Z20^4*Z2^20*Z30^3
+2^8*3^19*5^6*7^2*11177*205211 * Z20^5*Z12*Z2^19*Z30^2
+2^11*3^19*5^6*7^2*19*181213 * Z12^10*Z2^15*Z30^2
+2^6*3^18*5^6*7^2*191*567548017 * Z20^6*Z12^2*Z2^18*Z30
+2^7*3^19*5^6*7^2*18859*65519 * Z20*Z12^11*Z2^14*Z30
+2^7*3^19*5^8*7*263*347*30803 * Z20^7*Z12^3*Z2^17
+2^8*3^16*5^9*7^3*32820167 * Z20^2*Z12^12*Z2^13
+2^11*3^18*5^6*7^2*130881071 * Z12^7*Z2^18*Z30^3
+2^8*3^19*5^5*7^2*1229*1304753 * Z20*Z12^8*Z2^17*Z30^2
+2^6*3^18*5^5*7*13159*5789947 * Z20^7*Z2^20*Z30
-2^6*3^17*5^5*7^2*213751052731 * Z20^2*Z12^9*Z2^16*Z30
-2^10*3^17*5^10*7*11*1529413 * Z20^8*Z12*Z2^19
-2^9*3^16*5^6*7^2*173*127972109 * Z20^3*Z12^10*Z2^15
-2^9*3^19*5^7*7^2*37*97*4337 * Z12^4*Z2^21*Z30^4
+2^9*3^20*5^6*7^5*1018301 * Z20*Z12^5*Z2^20*Z30^3
+2^6*3^17*5^7*7^4*157*6334547 * Z20^2*Z12^6*Z2^19*Z30^2
+2^8*3^16*5^4*7^2*19*2417*2551*12037 * Z20^3*Z12^7*Z2^18*Z30
-2^7*3^17*5^5*7^2*89*191*17853503 * Z20^4*Z12^8*Z2^17
-2^10*3^21*5^8*7^2*2137 * Z12*Z2^24*Z30^5
-2^10*3^21*5^7*7^2*17*227*347 * Z20*Z12^2*Z2^23*Z30^4
+2^10*3^18*5^6*7^2*23*45166171 * Z20^2*Z12^3*Z2^22*Z30^3
+2^6*3^20*5^4*7^2*29*37*324827941 * Z20^3*Z12^4*Z2^21*Z30^2
+2^7*3^18*5^3*7^3*11*131*29027*48571 * Z20^4*Z12^5*Z2^20*Z30
+2^7*3^15*5^5*7^3*31*23099*1885459 * Z20^5*Z12^6*Z2^19
-2^10*3^17*5^5*7^3*827*14221 * Z12^15*Z2^15
-2^8*3^19*5^6*7^2*13928549 * Z20^2*Z2^25*Z30^4
+2^8*3^18*5^5*7^2*11*103*421*8707 * Z20^3*Z12*Z2^24*Z30^3
+2^12*3^19*5^4*7^3*13*47560897 * Z20^4*Z12^2*Z2^23*Z30^2
+2^7*3^16*5^4*7^2*13*719*1783*669241 * Z20^5*Z12^3*Z2^22*Z30
+2^7*3^14*5^5*7^2*548100080477 * Z12^12*Z2^18*Z30
+2^8*3^15*5^4*7^3*1780283701487 * Z20^6*Z12^4*Z2^21
+2^10*3^16*5^7*7^2*13*101*265459 * Z20*Z12^13*Z2^17
+2^7*3^17*5^2*7^2*709*7477*86131 * Z20^5*Z2^25*Z30^2
-2^16*3^16*5^6*7^2*13832521 * Z12^9*Z2^21*Z30^2
+2^6*3^16*5^4*7^2*13*23*4262123611 * Z20^6*Z12*Z2^24*Z30
+2^9*3^15*5^4*7^2*1033*1979*13001 * Z20*Z12^10*Z2^20*Z30
-2^10*3^18*5^4*7^2*32051*182159 * Z20^7*Z12^2*Z2^23
+2^8*3^14*5^4*7^3*11*19641133387 * Z20^2*Z12^11*Z2^19
+2^10*3^17*5^5*7^2*17*189716399 * Z12^6*Z2^24*Z30^3
+2^7*3^16*5^4*7^2*101*5777699071 * Z20*Z12^7*Z2^23*Z30^2
-2^8*3^16*5^3*7^2*4824842354177 * Z20^2*Z12^8*Z2^22*Z30
-2^9*3^16*5^2*7^2*13*21990109487 * Z20^8*Z2^25
-2^8*3^13*5^3*7^2*23*5623*1411172599 * Z20^3*Z12^9*Z2^21
-2^10*3^17*5^7*7^2*31*439667 * Z12^3*Z2^27*Z30^4
+2^9*3^17*5^5*7^3*13*19^2*173*2281 * Z20*Z12^4*Z2^26*Z30^3
+2^7*3^20*5^5*7^2*9184353733 * Z20^2*Z12^5*Z2^25*Z30^2
+2^12*3^15*5^3*7^2*2093780213323 * Z20^3*Z12^6*Z2^24*Z30
+2^9*3^14*5^2*7^2*29*101*33115141459 * Z20^4*Z12^7*Z2^23
-2^10*3^19*5^7*7^2*2137 * Z2^30*Z30^5
-2^9*3^19*5^6*7^2*19*321821 * Z20*Z12*Z2^29*Z30^4
+2^11*3^17*5^6*7^2*17*13839599 * Z20^2*Z12^2*Z2^28*Z30^3
+2^7*3^16*5^6*7^2*47*79^2*302009 * Z20^3*Z12^3*Z2^27*Z30^2
+2^7*3^15*5^2*7^3*29*101*367*3593*15277 * Z20^4*Z12^4*Z2^26*Z30
+2^9*3^15*7^2*11*5903837*31733797 * Z20^5*Z12^5*Z2^25
+2^9*3^13*5^5*7^3*569*977*1553 * Z12^14*Z2^21
+2^8*3^16*5^4*7^2*11*554789143 * Z20^3*Z2^30*Z30^3
+2^7*3^16*5^3*7^2*3307*5749*40763 * Z20^4*Z12*Z2^29*Z30^2
+2^10*3^18*5^2*7^2*13*283*613*122117 * Z20^5*Z12^2*Z2^28*Z30
+2^8*3^13*5^3*7^2*97*120517321303 * Z12^11*Z2^24*Z30
+2^8*3^15*5*7^2*2351*59753*315527 * Z20^6*Z12^3*Z2^27
+2^9*3^13*5^3*7^3*1019375392639 * Z20*Z12^12*Z2^23
-2^9*3^15*5^4*7^2*229*382783829 * Z12^8*Z2^27*Z30^2
+2^8*3^15*5^2*7^4*53*257*695689 * Z20^6*Z2^30*Z30
-2^14*3^13*5^6*7^2*79*761*46691 * Z20*Z12^9*Z2^26*Z30
-2^8*3^15*5*7^3*353*5954683957 * Z20^7*Z12*Z2^29
-2^10*3^12*5^2*7^2*29*59*149*157*1856221 * Z20^2*Z12^10*Z2^25
+2^10*3^18*5^3*7^3*359*3061027 * Z12^5*Z2^30*Z30^3
+2^8*3^15*5^3*7^3*340604380033 * Z20*Z12^6*Z2^29*Z30^2
-2^13*3^14*5^2*7^2*139*39383*152777 * Z20^2*Z12^7*Z2^28*Z30
-2^12*3^13*5*7^2*63378965258321 * Z20^3*Z12^8*Z2^27
-2^12*3^16*5^6*7^4*19*4483 * Z12^2*Z2^33*Z30^4
+2^11*3^16*5^5*7^4*41*563*1201 * Z20*Z12^3*Z2^32*Z30^3
+2^8*3^15*5^3*7^3*31*22664066651 * Z20^2*Z12^4*Z2^31*Z30^2
+2^10*3^14*5^2*7^3*541*569*16524979 * Z20^3*Z12^5*Z2^30*Z30
+2^9*3^13*5*7^3*41*3048495311723 * Z20^4*Z12^6*Z2^29
-2^10*3^16*5^5*7^3*19*56501 * Z20*Z2^35*Z30^4
+2^11*3^15*5^4*7^3*11*13*3002269 * Z20^2*Z12*Z2^34*Z30^3
+2^8*3^18*5^3*7^4*812189933 * Z20^3*Z12^2*Z2^33*Z30^2
+2^8*3^14*5^2*7^3*19*235679*1825937 * Z20^4*Z12^3*Z2^32*Z30
+2^9*3^14*5^2*7^3*149*18380304779 * Z20^5*Z12^4*Z2^31
+2^10*3^12*5^2*7^2*13*911*164646281 * Z12^13*Z2^27
+2^8*3^16*5^3*7^3*15919*40127 * Z20^4*Z2^35*Z30^2
+2^15*3^16*5^2*7^3*365440423 * Z20^5*Z12*Z2^34*Z30
+2^9*3^11*5*7^2*11*1399*23867817047 * Z12^10*Z2^30*Z30
-2^10*3^14*5*7^3*12653*77512949 * Z20^6*Z12^2*Z2^33
+2^10*3^12*5*7^2*6073*7561*1419713 * Z20*Z12^11*Z2^29
-2^9*3^12*5^3*7^2*41*673*101739863 * Z12^7*Z2^33*Z30^2
-2^12*3^12*5^2*7^2*177823*45043877 * Z20*Z12^8*Z2^32*Z30
-2^12*3^17*5*7^3*29^2*37*101*139 * Z20^7*Z2^35
-2^11*3^10*5^3*7^2*43*139*5977447793 * Z20^2*Z12^9*Z2^31
+2^11*3^12*5^4*7^3*13^2*70718519 * Z12^4*Z2^36*Z30^3
+2^9*3^16*5^3*7^3*9043*1393387 * Z20*Z12^5*Z2^35*Z30^2
+2^10*3^12*5^2*7^3*67763*24003367 * Z20^2*Z12^6*Z2^34*Z30
+2^12*3^10*5^3*7^3*223*2486599937 * Z20^3*Z12^7*Z2^33
-2^11*3^15*5^5*7^3*684121 * Z12*Z2^39*Z30^4
+2^12*3^13*5^4*7^3*73*353*36683 * Z20*Z12^2*Z2^38*Z30^3
+2^9*3^11*5^3*7^4*349021690219 * Z20^2*Z12^3*Z2^37*Z30^2
+2^11*3^10*5^2*7^3*11*587*3853542631 * Z20^3*Z12^4*Z2^36*Z30
+2^10*3^11*5^2*7^4*11*198879218137 * Z20^4*Z12^5*Z2^35
+2^8*3^12*5^4*7^4*41*419*15139 * Z20^2*Z2^40*Z30^3
+2^9*3^11*5^3*7^3*170173386991 * Z20^3*Z12*Z2^39*Z30^2
+2^9*3^11*5^2*7^3*47*8731*12540581 * Z20^4*Z12^2*Z2^38*Z30
-2^10*3^9*5^2*7^3*631*16453*714151 * Z20^5*Z12^3*Z2^37
+2^11*3^12*5^2*7^3*19*3630744617 * Z12^12*Z2^33
+2^11*3^10*5*7^4*11*3717589009 * Z20^5*Z2^40*Z30
+2^10*3^8*5^4*7^3*17*643*719*109903 * Z12^9*Z2^36*Z30
-2^11*3^9*5*7^3*4133*7243*304907 * Z20^6*Z12*Z2^39
+2^11*3^10*5^2*7^3*11*12131767093 * Z20*Z12^10*Z2^35
-2^10*3^9*5^3*7^3*83*29297*268283 * Z12^6*Z2^39*Z30^2
-2^14*3^11*5^2*7^3*139491370861 * Z20*Z12^7*Z2^38*Z30
-2^12*3^9*5^2*7^3*17^2*19861272577 * Z20^2*Z12^8*Z2^37
+2^18*3^10*5^4*7^3*1619*44357 * Z12^3*Z2^42*Z30^3
+2^12*3^11*5^3*7^3*151*282910193 * Z20*Z12^4*Z2^41*Z30^2
+2^10*3^10*5^2*7^3*173*9311*2585293 * Z20^2*Z12^5*Z2^40*Z30
+2^12*3^7*5^2*7^3*45834651100351 * Z20^3*Z12^6*Z2^39
-2^12*3^12*5^4*7^3*684121 * Z2^45*Z30^4
+2^11*3^12*5^4*7^3*71*83*179*193 * Z20*Z12*Z2^44*Z30^3
+2^10*3^11*5^3*7^4*31*37*1019*6521 * Z20^2*Z12^2*Z2^43*Z30^2
+2^11*3^8*5^2*7^3*139*101441196979 * Z20^3*Z12^3*Z2^42*Z30
+2^11*3^9*5^2*7^3*351151*6647521 * Z20^4*Z12^4*Z2^41
+2^14*3^9*5^3*7^3*1327*309251 * Z20^3*Z2^45*Z30^2
+2^12*3^10*5^2*7^3*17*41*4937*13249 * Z20^4*Z12*Z2^44*Z30
-2^11*3^10*5^2*7^4*50776597931 * Z20^5*Z12^2*Z2^43
+2^12*3^7*5^4*7^4*29*43*281*31153 * Z12^11*Z2^39
+2^13*3^7*5^3*7^3*109*181*7354283 * Z12^8*Z2^42*Z30
-2^12*3^8*5^2*7^3*607*90804557 * Z20^6*Z2^45
-2^12*3^7*5^3*7^3*227*10651*279967 * Z20*Z12^9*Z2^41
-2^15*3^8*5^2*7^3*53^2*7401907 * Z12^5*Z2^45*Z30^2
-2^12*3^7*5^2*7^3*17*235269698801 * Z20*Z12^6*Z2^44*Z30
-2^12*3^6*5^2*7^3*33851*133712119 * Z20^2*Z12^7*Z2^43
+2^15*3^9*5^4*7^3*163*179*2917 * Z12^2*Z2^48*Z30^3
+2^11*3^8*5^3*7^3*12577*11233711 * Z20*Z12^3*Z2^47*Z30^2
+2^12*3^7*5^3*7^3*13*29^2*59*850691 * Z20^2*Z12^4*Z2^46*Z30
+2^12*3^7*5^3*7^3*11*82987837567 * Z20^3*Z12^5*Z2^45
+2^11*3^9*5^4*7^3*41*1902613 * Z20*Z2^50*Z30^3
+2^11*3^8*5^3*7^3*311*3449*13781 * Z20^2*Z12*Z2^49*Z30^2
+2^12*3^7*5^4*7^4*35851*115807 * Z20^3*Z12^2*Z2^48*Z30
-2^13*3^8*5^3*7^3*31084535837 * Z20^4*Z12^3*Z2^47
+2^12*3^7*5^3*7^3*367*6651413 * Z20^4*Z2^50*Z30
-2^13*3^6*5^4*7^3*11*1196570233 * Z20^5*Z12*Z2^49
+2^16*3^6*5^4*7^4*254505289 * Z12^10*Z2^45
+2^12*3^5*5^3*7^3*8971*15690341 * Z12^7*Z2^48*Z30
-2^14*3^6*5^3*7^5*2411*769799 * Z20*Z12^8*Z2^47
-2^12*3^6*5^4*7^3*53*1019*43013 * Z12^4*Z2^51*Z30^2
-2^13*3^6*5^3*7^3*153073*402817 * Z20*Z12^5*Z2^50*Z30
+2^16*3^4*5^3*7^3*2017*60384263 * Z20^2*Z12^6*Z2^49
+2^13*3^8*5^4*7^3*27165209 * Z12*Z2^54*Z30^3
+2^12*3^7*5^3*7^3*18637*381001 * Z20*Z12^2*Z2^53*Z30^2
+2^12*3^5*5^3*7^3*113*2195698811 * Z20^2*Z12^3*Z2^52*Z30
+2^14*3^4*5^3*7^3*199*1061*743221 * Z20^3*Z12^4*Z2^51
+2^13*3^6*5^3*7^3*227639521 * Z20^2*Z2^55*Z30^2
+2^14*3^5*5^3*7^3*11*17*31217027 * Z20^3*Z12*Z2^54*Z30
-2^15*3^7*5^3*7^3*11689*201767 * Z20^4*Z12^2*Z2^53
-2^19*3^5*5^4*7^4*659*1979 * Z20^5*Z2^55
+2^14*3^5*5^4*7^3*11*4933*28661 * Z12^9*Z2^51
+2^16*3*5^6*7^3*179854793 * Z12^6*Z2^54*Z30
-2^14*3^4*5^4*7^3*13463475413 * Z20*Z12^7*Z2^53
-2^13*3^2*5^6*7^3*11*31*226669 * Z12^3*Z2^57*Z30^2
-2^16*3^2*5^4*7^3*41*2311*38767 * Z20*Z12^4*Z2^56*Z30
+2^14*3*5^4*7^3*37*41*397*533879 * Z20^2*Z12^5*Z2^55
+2^12*3^6*5^3*7^3*27165209 * Z2^60*Z30^3
+2^13*3^7*5^3*7^3*167*350503 * Z20*Z12*Z2^59*Z30^2
+2^15*3^2*5^3*7^3*1619*2861*4909 * Z20^2*Z12^2*Z2^58*Z30
-2^14*5^3*7^3*29*283*110896741 * Z20^3*Z12^3*Z2^57
+2^15*3*5^4*7^4*1289*39023 * Z20^3*Z2^60*Z30
-2^15*3*5^4*7^4*11*13*11210723 * Z20^4*Z12*Z2^59
+2^16*5^6*7^4*23055827 * Z12^8*Z2^57
+2^18*3*5^6*7^4*13*8233 * Z12^5*Z2^60*Z30
-2^16*5^5*7^6*59*120233 * Z20*Z12^6*Z2^59
-2^14*3^2*5^4*7^4*1151659 * Z12^2*Z2^63*Z30^2
-2^17*3*5^4*7^4*17^2*47*883 * Z20*Z12^3*Z2^62*Z30
+2^16*3*5^4*7^4*29*19709777 * Z20^2*Z12^4*Z2^61
+2^14*3^2*5^4*7^4*1151659 * Z20*Z2^65*Z30^2
+2^17*3^2*5^4*7^4*11*31*43*151 * Z20^2*Z12*Z2^64*Z30
-2^16*3^3*5^4*7^4*19456027 * Z20^3*Z12^2*Z2^63
-2^17*3^2*5^5*7^5*11*4021 * Z20^4*Z2^65
+2^17*3^2*5^7*7^5*853 * Z12^7*Z2^63
+2^16*3^3*5^5*7^5*13*17 * Z12^4*Z2^66*Z30
-2^17*3^2*5^5*7^5*64793 * Z20*Z12^5*Z2^65
-2^17*3^3*5^5*7^5*13*17 * Z20*Z12^2*Z2^68*Z30
+2^17*3^2*5^5*7^7*13*103 * Z20^2*Z12^3*Z2^67
+2^16*3^3*5^5*7^5*13*17 * Z20^2*Z2^70*Z30
-2^17*3^3*5^5*7^5*11^2*61 * Z20^3*Z12*Z2^69
+2^18*3^3*5^6*7^6 * Z12^6*Z2^69
-2^18*3^4*5^6*7^6 * Z20*Z12^4*Z2^71
+2^18*3^4*5^6*7^6 * Z20^2*Z12^2*Z2^73
-2^18*3^3*5^6*7^6 * Z20^3*Z2^75
)
\end{verbatim}
}

\end{theorem}

\subsection{The discriminants of ${\cal W}(H_4)$ and ${\cal W}^{\ast}(H_4)$}

Let ${\cal R}^{\ast}$ be the set $\{\varphi^{\ast}|\>\varphi\in{\cal R}\}$,
where $\varphi^{\ast}=\varphi|_{\sqrt{5}\to-\sqrt{5}}$.
Using ${\cal R}^{\ast}$, we define $D^{\ast}=\prod_{\varphi\in{\cal R}^{\ast}}\varphi$.

\begin{lemma}
For each $\varphi\in{\cal R}^{\ast}$,
$
\varphi(P_1,P_2,P_3,P_4)$ has the factor $\varphi^{\ast}$.
\end{lemma}

{\bf Proof}. Since $\varphi_1=u_1$ is contained in ${\cal R}$, it follows
that $\varphi_1^{\ast}=P_1$.
Since $P_1/u_1$ is a polynomial, the action of ${\cal W}^{\ast}(H_4)$
shows that for any $\varphi\in{\cal R}$,
$\varphi^{\ast}(P_1,P_2,P_3,P_4)/\varphi$ is a polynomial of ${\tt u}$.
[]

\vspace{5mm}

We put
$$
G_{\varphi}(u_1,u_2,u_3,u_4)=\varphi(P_1,P_2,P_3,P_4)/\varphi \quad(\varphi\in{\cal R}^{\ast}).
$$
Since  each $P_j$ is a homogeneous polynomial of degree 7, 
$G_{\varphi}$ is a homogeneous polynomial of degree 6.

\begin{lemma}
$$\frac{D^{\ast}(P_1,P_2,P_3,P_4)}{D({\tt u})}
=\prod_{\varphi\in{\cal R}}G_{\varphi^{\ast}}({\tt u}).
$$
\end{lemma}

{\bf Proof}.
By definition,
we have
$$
D({\tt u})=\prod_{\varphi\in{\cal R}}\varphi({\tt u})
$$
and
$$
D^{\ast}(P_1,P_2,P_3,P_4)=\prod_{\varphi\in{\cal R}}\varphi^{\ast}(P_1,P_2,P_3,P_4).
$$
Then the lemma follows.[]

\subsection{Jacobian matrices}

The purpose of this section is to compute the determinant of the matrix
$$
M=\frac{\partial(Y_2,Y_{12},Y_{20},Y_{30})}{\partial(u_1,u_2,u_3,u_4)}
$$
in two ways by using the chain rule of transformations.
It follows from the definition that
$$
M=\frac{\partial(P_1,P_2,P_3,P_4)}{\partial(u_1,u_2,u_3,u_4)}\frac{\partial(Y_2,Y_{12},Y_{20},Y_{30})}{\partial(P_1,P_2,P_3,P_4)}
=
\frac{\partial(Z_2,Z_{12},Z_{20},Z_{30})}{\partial(u_1,u_2,u_3,u_4)}\frac{\partial(Y_2,Y_{12},Y_{20},Y_{30})}{\partial((Z_2,Z_{12},Z_{20},Z_{30})}.
$$
It is known that
$$
\det\left(\frac{\partial(Z_2,Z_{12},Z_{20},Z_{30})}{\partial(u_1,u_2,u_3,u_4)})\right)=
\prod_{\varphi\in{\cal R}}\varphi=D(u_1,u_2,u_3,u_4).
$$
This is a formula concerning $D(u_1,u_2,u_3,u_4)$ for ${\cal W}(H_4)$.
$$
\det\left(\frac{\partial(Y_2,Y_{12},Y_{20},Y_{30})}{\partial(P_1,P_2,P_3,P_4)}\right)=
\prod_{\varphi\in{\cal R}^{\ast}}\varphi=D^{\ast}(P_1,P_2,P_3,P_4).
$$
On the other hand, it is straightforward to show that
$\det\left(\frac{\partial(P_1,P_2,P_3,P_4)}{\partial(u_1,u_2,u_3,u_4)}\right)$
coincides with $Z_2^2Z_{20}-Z_{12}^2$ up to a constant factor.
Since
$$
\begin{array}{ll}
&\det\left(\frac{\partial(P_1,P_2,P_3,P_4)}{\partial(u_1,u_2,u_3,u_4)}\right)
\det\left(\frac{\partial(Y_2,Y_{12},Y_{20},Y_{30})}{\partial(P_1,P_2,P_3,P_4)}\right)\\
=&
\det\left(\frac{\partial(Z_2,Z_{12},Z_{20},Z_{30})}{\partial(u_1,u_2,u_3,u_4)}\right)\det\left(\frac{\partial(Y_2,Y_{12},Y_{20},Y_{30})}{\partial(Z_2,Z_{12},Z_{20},Z_{30})}\right),\\
\end{array}
$$
we find that
$$
\det\left(\frac{\partial(Y_2,Y_{12},Y_{20},Y_{30})}{\partial(Z_2,Z_{12},Z_{20},Z_{30})}\right)
=c'(Z_2^2Z_{20}-Z_2^2)\prod_{\varphi\in{\cal R}^{\ast}}G_{\varphi}.
$$
for a non-zero constant $c'$.
In \cite{Ar}, through direct computation using a computer algebra system, 
the author verifies the above equation and determines the constant to be $c' = 2^{10} \cdot 3^2$.

\subsection{Prepotentials related with $H_4$}

We introduce two  prepotentials related to the group $W(H_4)$.
One is the polynomial prepotential denoted by $F_{H_4}$
and the other is $F_{H_4(9)}$ which 
is an analogue of the algebraic  potential  $F_{(H_3)'}$.

The polynomial potential $F_{H_4}$ is given in \cite{Du}, Lecture 4, Example 4.4:
\begin{equation}
\begin{array}{lll}
F_{H_4}&=&\frac{{x_1} {x_4}^2}{2}+{x_2} {x_3} {x_4}+\frac{32 }{22395255890625}x_1{}^{31}+
\frac{2}{15582375}{x_1}^{19}{x_2}^2+\frac{1}{72900}{x_1}^{13}{x_2}^3\\
&&+\frac{8 }{22275}{x_1}^{11}{x_3}^2
  +\frac{1}{810} {x_1}^9 {x_2}^2 {x_3}+\frac{1}{1080}{x_1}^7{x_2}^4
   +\frac{1}{15} {x_1}^5 {x_2} {x_3}^2+\frac{1}{18} {x_1}^3
   {x_2}^3 {x_3}\\
   &&+\frac{1}{240}{x_1}{x_2}^5+\frac{2}{3} {x_1}
   {x_3}^3.\\
   \end{array}
\end{equation}
Let $C_{H_4}$ be the $4\times 4$ matrix such that
$(C_{H_4})_{ij}=\partial_{x_i}\partial_{x_{5-j}}F_{H_4}$ and
$$
T_{H_4}=\frac{1}{30}(2x_1\partial_{x_1}C_{H_4}
+12x_2\partial_{x_{2}}C_{H_4}+20x_3\partial_{x_3}C_{H_4}+30x_4\partial_{x_4}C_{H_4}).
$$
Then $\Delta_{H_4}=\det T_{H_4}$ is called the discriminant of $F_{H_4}$.
Its concrete form is given by

{\small
$$
\begin{array}{ll}
&\Delta_{H_4}\\
=&65536x_1^{60} + 33177600x_1^{54}x_2 + 1492992000x_1^{48}x_2^{2 }- 
   596263680000x_1^{42}x_2^{3} \\
&- 35287682400000x_1^{36}x_2^{4} + 
   3697075604700000x_1^{30}x_2^{5 }- 251572211977500000x_1^{24}x_2^{6} \\
&- 
   2217376955010937500x_1^{18}x_2^{7} + 
 145866084150427734375x_1^{12}x_2^{8} \\
&- 
   2273083783917539062500x_1^{6}x_2^{9} + 397166535676406250000x_2^{10 }+ 
   995328000x_1^{50}x_3\\
& + 67184640000x_1^{44}x_2x_3 + 
 31744742400000x_1^{38}x_2^{2}x_3 + 
   684496008000000x_1^{32}x_2^{3}x_3 \\
&+ 
 483199443225000000x_1^{26}x_2^{4}x_3 + 
   27221132108362500000x_1^{20}x_2^{5}x_3 \\
&+ 
 92308984949390625000x_1^{14}x_2^{6}x_3 + 
   7938427422964218750000x_1^{8}x_2^{7}x_3 \\
&- 
 14562772974801562500000x_1^{2}x_2^{8}x_3 - 
   8062156800000x_1^{40}x_3^{2} \\
&+ 1666598976000000x_1^{34}x_2x_3^{2} - 
   350113330800000000x_1^{28}x_2^{2}x_3^{2} \\
&- 
 24924051459000000000x_1^{22}x_2^{3}x_3^{2} - 
   1549439818194375000000x_1^{16}x_2^{4}x_3^{2} \\
&- 
 34166128649298750000000x_1^{10}x_2^{5}
     x_3^{2} - 30964279911067968750000x_1^{4}x_2^{6}x_3^{2}\\
& - 
   32311612800000000x_1^{30}x_3^{3} - 585435405600000000x_1^{24}x_2x_3^{3} \\
&- 
   730675042254000000000x_1^{18}x_2^{2}x_3^{3} - 
 15354926805903750000000x_1^{12}x_2^{3}
     x_3^{3}\\
&- 576332772881540625000000x_1^{6}x_2^{4}x_3^{3} - 
   105911076180375000000000x_2^{5}x_3^{3} \\
&+ 
 146703225168000000000x_1^{20}x_3^{4} + 
   23134814500635000000000x_1^{14}x_2x_3^{4} \\
&- 
 109441445386387500000000x_1^{8}x_2^{2}
     x_3^{4} - 3353850745711875000000000x_1^{2}x_2^{3}x_3^{4} \\
&- 
   237763828304190000000000x_1^{10}x_3^{5} - 
 4942516888417500000000000x_1^{4}x_2
     x_3^{5} \\
&+ 7060738412025000000000000x_3^{6} \\
&- 
   3690562500
  x_1(512x_1^{38}x_2 + 34560x_1^{32}x_2^{2} + 14580000x_1^{26}x_2^{3} + 
        1673055000x_1^{20}x_2^{4} \\
&+ 12178856250x_1^{14}x_2^{5} + 
    1343549278125x_1^{8}x_2^{6} - 
        9267002437500x_1^{2}x_2^{7} + 27648x_1^{34}x_3\\
& - 
    12441600x_1^{28}x_2x_3 - 
        1522152000x_1^{22}x_2^{2}x_3 - 184823370000x_1^{16}x_2^{3}x_3 \\
&- 
        6995092162500x_1^{10}x_2^{4}x_3 - 29235898012500x_1^{4}x_2^{5}x_3 - 
        45664560000x_1^{18}x_2x_3^{2}\\
& - 1753755300000x_1^{12}x_2^{2}x_3^{2} - 
        212842120500000x_1^{6}x_2^{3}x_3^{2} - 430467210000000x_2^{4}x_3^{2} \\
&+ 
        1275458400000x_1^{14}x_3^{3} - 37732311000000x_1^{8}x_2x_3^{3 }- 
        3085015005000000x_1^{2}x_2^{2}x_3^{3} \\
&- 1090516932000000x_1^{4}x_3^{4})
  x_4 \\
&- 
   5978711250000(128x_1^{30} + 32400x_1^{24}x_2 + 8748000x_1^{18}x_2^{2} + 
        586389375x_1^{12}x_2^{3}\\
&+ 7750181250x_1^{6}x_2^{4} + 
    9964518750x_2^{5} + 
        972000x_1^{20}x_3 + 65610000x_1^{14}x_2x_3 \\
&+ 
    36536568750x_1^{8}x_2^{2}x_3 + 
        481618406250x_1^{2}x_2^{3}x_3 + 3936600000x_1^{10}x_3^{2} \\
&+ 
        1129312125000x_1^{4}x_2x_3^{2} + 1328602500000x_3^{3})x_4^{2} \\
&+ 
   5516201884394531250000x_1^{3}(2x_1^{6}x_2 + 135x_2^{2} + 108x_1^{2}x_3)
  x_4^{3} \\
&+ 
   2234061763179785156250000x_4^{4}.\\
\end{array}
$$
}

The following lemma is shown by direct computation.

\begin{lemma}
Let $D(u)$ be the product of linear forms in (\ref{equation:60-linear forms}).
Then by the identification
\begin{equation}
\left\{
\label{equation:map-x-Z}
\begin{array}{lll}
x_1 &=& Z_2, \\
x_2 &=& \frac{2}{45} (45 Z_{12} + 4 Z_2^6), \\
 x_3 &=& \frac{1}{405} (135 Z_{12} Z_2^4 + 8 Z_2^{10 }+ 405 Z_{20}), \\
   x_4 &=& \frac{1}{18225}(12150 Z_{12}^2 Z_2^3 + 2520 Z_{12} Z_2^9 + 104 Z_2^{15} + 
     4860 Z_2^5 Z_{20} - 18225 Z_{30}),\\
     \end{array}
     \right.
\end{equation}
$\Delta_{H_4}(x_1,x_2,x_3,x_4)$ coincides with ${\tilde D}(Z_2,Z_{12},Z_{20},Z_{30})$ (cf. 
 (\ref{equation:def-DFZ}))
 up to a constant factor.
\end{lemma}

There are  seven algebraic Frobenius potentials related $H_4$, named $H_4(j)\>(j=1,2,3,4,6,7,9)$
(cf. \cite{Se}).
Among others we focus our attention on the algebraic potential $F_{H_4(9)}$ which is given as follows:
{\footnotesize
$$
\begin{array}{lll}
F_{H_4(9)}&=&\frac{1}{2}t_1 t_4^2+t_2 t_3 t_4 -\frac{ 189 t_1^6}{800 w_0^5} -\frac{ 63 t_1^4 t_2}{40 w_0}+
\frac{1}{50}t_1 t_2^3 - \frac{189}{32} t_1^5 w_0^2
 - \frac{63}{50} t_1^2 t_2^2 w_0^3 - 
   \frac{273}{16} t_1^3 t_2 w_0^6\\
   && -\frac{2}{25} t_2^3 w_0^7 -\frac{3017}{52} t_1^4 w_0^9
 - \frac{217}{100} t_1 t_2^2 w_0^{10} - \frac{1988}{65} t_1^2 t_2 w_0^{13 }- 
  \frac{53865}{416} t_1^3 w_0^{16} - \frac{308}{221} t_2^2 w_0^{17}\\
  && - \frac{20321}{780}t_1 t_2 
  w_0^{20} - \frac{50246}{345} t_1^2 w_0^{23} - \frac{4984}{585} t_2 w_0^{27} - 
   \frac{16324}{195} t_1 w_0^{30} - \frac{9072}{481} w_0^{37},\\
   \end{array}
   $$
}   
where $w_0$ is an algebraic function of
$t_1,t_2,t_3$ defined by the equation
\begin{equation}
\label{equation:def-eq-w0-t3}
\frac{3}{20} {t_1}^2-{t_3}w_0^2+\frac{1}{5}{t_2} w_0^4+\frac{8}{5} {t_1}
   w_0^7+w_0^{14}=0.
 \end{equation}
Similarly to the case $F_{H_4}$, we can define 
the discriminant $\Delta_{H_4(9)}=\Delta_{H_4(9)}(t_1,t_2,t_3,t_4)$
of $F_{H_4(9)}$.
It follows from the equation (\ref{equation:def-eq-w0-t3}) that
$t_3$ is regarded as a function of $t_1,t_2,w_0$
defined by
\begin{equation}
\label{equation:def-eq-t3-w0}
t_3=(\frac{3}{20} {t_1}^2+\frac{1}{5}{t_2} w_0^4+\frac{8}{5} {t_1}
   w_0^7+w_0^{14})/w_0^2.
 \end{equation}
Eliminating  $t_3$ in $\Delta_{H_4(9)}$  by this relation, we find that
$\Delta_{H_4(9)}$ is regarded as a function of $t_1,t_2,t_4,w_0$.
So we put
$$
\Psi_{H_4(9)}(t_1,t_2,t_4,w_0)=\Delta_{H_4(9)}(t_1,t_2,(\frac{3}{20} {t_1}^2+\frac{1}{5}{t_2} w_0^4+\frac{8}{5} {t_1}
   w_0^7+w_0^{14})/w_0^2,t_4).
   $$
Then the concrete form of ${\tilde \Psi}_{H_4(9)}=72\cdot 10^6\cdot w_0^{10}\Psi_{H_4(9)}$
is given by
{\small
$$
\begin{array}{ll}
&{\tilde \Psi}_{H_4(9)}\\
=&23147208t_1^{10} - 282910320t_1^8t_2w_0^4 + 4354560t_1^5t_2^3w_0^5 + 
   11010105225t_1^9w_0^7 
- 657742680t_1^6t_2^2w_0^8\\
& + 
 45964800t_1^3t_2^4w_0^9 + 
   204800t_2^6w_0^{10} 
 - 
 3237513300t_1^7t_2w_0^{11} - 
   1343835360t_1^4t_2^3w_0^{12}\\
& + 33331200t_1t_2^5w_0^{13} + 
   131363382150t_1^8w_0^{14} - 12777322320t_1^5t_2^2w_0^{15} + 
   290720640t_1^2t_2^4w_0^{16}\\
& + 106165325700t_1^6t_2w_0^{18} + 
   1244196800t_1^3t_2^3w_0^{19} - 135657984t_2^5w_0^{20} + 
   1257858456750t_1^7w_0^{21} \\
&+ 61206801600t_1^4t_2^2w_0^{22} - 
   4971993600t_1t_2^4w_0^{23} + 840446451600t_1^5t_2w_0^{25} - 
   82051435200t_1^2t_2^3w_0^{26} \\
&+ 3980430450500t_1^6w_0^{28} - 
   589067808000t_1^3t_2^2w_0^{29} - 3231513600t_2^4w_0^{30} - 
   1325951907000t_1^4t_2w_0^{32}\\
&- 97451760000t_1t_2^3w_0^{33}+ 
   1846006191750t_1^5w_0^{35} - 1134341964000t_1^2t_2^2w_0^{36}\\
& - 
   5338498305000t_1^3t_2w_0^{39} - 29148672000t_2^3w_0^{40} - 
   7358004701250t_1^4w_0^{42}\\
& - 665874720000t_1t_2^2w_0^{43} - 
   5073299280000t_1^2t_2w_0^{46} - 11947799850000t_1^3w_0^{49 }\\
&- 
   128701440000t_2^2w_0^{50}
 - 1979208000000t_1t_2w_0^{53} - 
   7510935600000t_1^2w_0^{56} \\
&- 282240000000t_2w_0^{60} - 
 2194416000000t_1w_0^{63} - 
   248371200000w_0^{70} \\
&+ 
   4200t_4
  w_0^6(629370t_1^7 - 51840t_1^4t_2^2w_0
+ 1346058t_1^5t_2w_0^4 - 
        133440t_1^2t_2^3w_0^5 + 14189175t_1^6w_0^7 \\
&- 
    133200t_1^3t_2^2w_0^8 - 
        19712t_2^4w_0^9 
+ 14030100t_1^4t_2w_0^{11} - 
    102880t_1t_2^3w_0^{12} + 
        63590730t_1^5w_0^{14} \\
&+ 4470000t_1^2t_2^2w_0^{15} + 
    65580200t_1^3t_2w_0^{18} 
- 
        228800t_2^3w_0^{19} + 237288625t_1^4w_0^{21} + 
    2673600t_1t_2^2w_0^{22} \\
&+ 
        77038500t_1^2t_2w_0^{25} + 358447250t_1^3w_0^{28}
 - 
    489600t_2^2w_0^{29} + 
        29232000t_1t_2w_0^{32} + 240345000t_1^2w_0^{35}\\
& + 
    2016000t_2w_0^{39} + 
        70560000t_1w_0^{42} 
+ 6720000w_0^{49})\\
& + 
 12000t_4^2w_0^5(6804t_1^5 + 109620t_1^3t_2w_0^4 - 
        640t_2^3w_0^5 + 207900t_1^4w_0^7 + 68880t_1t_2^2w_0^8 \\
&+ 
        804300t_1^2t_2w_0^{11} + 1623475t_1^3w_0^{14 }- 
    5040t_2^2w_0^{15} + 
        384300t_1t_2w_0^{18 }+ 1008000t_1^2w_0^{21}\\
&
 - 100800t_2w_0^{25} - 
        693000t_1w_0^{28} - 504000w_0^{35}) \\
&- 
 25200000t_4^3w_0^{11}(15t_1^2 - 4t_2w_0^4 - 50t_1w_0^7 - 
        40w_0^{14})\\
&+ 72000000t_4^4w_0^{10}\\
\end{array}
$$
}

As is shown in \cite{Se}, \S5.3, {\bf (Case 1)}, by the transformation
\begin{equation}
\label{equation:map-t-x}
\left\{\begin{array}{lll}
t_1&=&\frac{1}{6}x_1(4{x_1}^6+315 x_2),\\
t_2&=&\frac{3675}{8}({x_1}^4x_2+18x_3),\\
t_4&=&\frac{46305}{32}(5{x_1}^3{x_2}^2+8{x_1}^5x_3+90x_4),\\
w_0&=&-x_1,\\
\end{array}
\right.
\end{equation}
$\Psi_{H_4(9)}(t_1,t_2,t_4,w_0)$
coincides with
$\Delta_{H_4}(x_1,x_2,x_3,x_4)$
up to a constant factor.

\begin{remark}
There is a mistake in the definition of the transformation 
given in \cite{Se}, \S5.3.
``$y_1=\frac{1}{6}x_1(4{x_1}^6+31 x_2)$'' should  read ``$y_1=\frac{1}{6}x_1(4{x_1}^6+315 x_2)$''.
\end{remark}

Composing the transformations  (\ref{equation:map-t-x}),
(\ref{equation:map-x-Z}), we have the following.

\begin{lemma}
\label{lemma:map-t-x-Z,mapZ-x-t}
(i) By the transformation
\begin{equation}
\label{equation:map-t-x-Z}
\left\{
\begin{array}{lll}
t_1 &=& 5 Z_2 (21 Z_{12} + 2 Z_2^6), \\
 t_2 &=& \frac{245}{4} (60 Z_{12} Z_2^4 + 4 Z_2^{10} + 135 Z_{20}), \\
   t_4 &=& \frac{343}{16} (5400 Z_{12}^2 Z_2^3 + 1260 Z_{12} Z_2^9 + 56 Z_2^{15} + 
     2160 Z_2^5 Z_{20} - 6075 Z_{30}),\\
      w_0 &=& -Z_2,\\
\end{array}
\right.
\end{equation}
$\Psi_{H_4(9)}(t_1,t_2,t_4,w_0)$
coincides with
${\tilde D}(Z_2,Z_{12},Z_{20},Z_{30})$
up to a constant factor.

(ii)
The inverse of
(\ref{equation:map-t-x-Z}) is given by
\begin{equation}
\label{equation:map-Z-x-t}
\left\{
\begin{array}{lll}
Z_2 &=& -w_0,\\
 Z_{12} &=& (-t_1 - 10 w_0^7)/(105 w_0), \\
  Z_{20} &=&\frac{4}{33075} (t_2 + 35 t_1 w_0^3 + 105 w_0^{10}), \\
     Z_{30} &=& -\frac{4}{10418625} (20 t_4 + 210 t_1^2 w_0 + 112 t_2 w_0^5 + 2975 t_1 w_0^8 + 
          5320 w_0^{15}).\\
          \end{array}\right.
          \end{equation}          
\end{lemma}

Combining Theorem \ref{theorem:formula-Y-Z} and (\ref{equation:map-Z-x-t}),
we can obtain formulas for the expression of $Y_j\>(j=2,12,20,30)$
as functions of  $t_1,t_2,t_4,w_0$.
In particular we find that
$Y_2,\>w_0^7Y_{12},\>w_0^{10}Y_{20},\>w_0^{15}Y_{30}$
are polynomials of $t_1,t_2,t_4,w_0$.
In spite that it is the final goal to show the 
concrete expressions of $Y_2,\>w_0^7Y_{12},\>w_0^{10}Y_{20},\>w_0^{15}Y_{30}$
as polynomials,
we abandon to include
their concrete forms in this paper 
by the reason that their expressions are very lengthy.


\end{document}